\documentclass{ws-book9x6-ppn}
\usepackage{ws-book-thm}
\usepackage{ws-book-har}
\usepackage{setspace,xspace}
\usepackage[pdfpagelabels=false]{hyperref}
\usepackage{cancel}
\usepackage{color}
\renewcommand\thesection{\arabic{section}}
\renewcommand\theequation{\arabic{equation}}

\newtheorem{Cor}[theorem]{Corollary}
\newtheorem{Lem}[theorem]{Lemma}
\newtheorem{Prop}[theorem]{Proposition}

\newcommand{\R}{{\mathbb R}}
\renewcommand{\S}{{\mathbb S}}

\renewcommand{\L}{\mathrm L}
\renewcommand{\H}{\mathrm H}
\newcommand{\be}[1]{\begin{equation}\label{#1}}
\newcommand{\ee}{\end{equation}}
\renewcommand{\(}{\left(}
\renewcommand{\)}{\right)}
\newcommand{\iS}[1]{\int_{\S^2}{#1}\,d\sigma}
\newcommand{\iC}[1]{\int_{\mathcal C}{#1}\,dy}
\newcommand{\irdeux}[1]{\int_{\R^2}{#1}\,dx}
\newcommand{\irdmu}[1]{\int_{\R^2}{#1}\,d\mu}

\newcommand{\nrm}[2]{\|{#1}\|_{\L^{#2}(\R^2)}}
\newcommand{\nrmS}[2]{\|{#1}\|_{\L^{#2}(\S^2)}}
\newcommand{\Lap}{\Delta_{\S^2}}

\newcommand{\scal}[2]{\left\langle{#1},{#2}\right\rangle}

\newcommand{\LL}{{\mathcal L}\,}
\newcommand{\izp}[1]{\int_{-1}^1{#1}\;dz}

\newcommand{\nr}[2]{\|{#1}\|_{\L^{#2}(\R^d)}}
\newcommand{\ird}[1]{\int_{\R^d}{#1}\;dx}

\newcommand{\nrml}[2]{\|{#1}\|_{\L^{#2}(\R)}}
\newcommand{\irl}[1]{\int_{\R}{#1}\;ds}

\begin{document}
\title{The Onofri inequality}
\thispagestyle{plain}
\pagestyle{plain}
\centerline{\Large{\bf The Moser-Trudinger-Onofri inequality}}
\vskip 1cm
\centerline{\scshape Jean Dolbeault, Maria J.~Esteban, Gaspard Jankowiak}
\vskip 0.5cm
{\footnotesize
\centerline{Ceremade, CNRS UMR n$^{\circ}$ 7534 and Universit\'e Paris-Dauphine}
\centerline{
Place de Lattre de Tassigny, 75775 Paris C\'edex~16, France}
\centerline{E-mail: \textsf{\href{mailto:dolbeaul@ceremade.dauphine.fr}{dolbeaul}, \href{mailto:esteban@ceremade.dauphine.fr}{esteban}, \href{mailto:jankowiak@ceremade.dauphine.fr}{jankowiak@ceremade.dauphine.fr}}}
\vskip 0.5cm
\centerline{\today}}\vspace*{1cm}
\begin{spacing}{0.9}\begin{quote}{\normalfont\fontsize{10}{12}\selectfont{\bfseries Abstract.}
This paper is devoted to results on the Moser-Trudinger-Onofri inequality, or Onofri inequality for brevity. In dimension two this inequality plays a role similar to the Sobolev inequality in higher dimensions. After justifying this statement by recovering the Onofri inequality through various limiting procedures and after reviewing some known results, we state several elementary remarks.

\hspace*{12pt}We also prove various new results. We give a proof of the inequality using mass transportation methods (in the radial case), consistently with similar results for Sobolev's inequalities. We investigate how duality can be used to improve the Onofri inequality, in connection with the logarithmic Hardy-Littlewood-Sobolev inequality. In the framework of fast diffusion equations, we establish that the inequality is an entropy--entropy production inequality, which provides an integral remainder term. Finally we give a proof of the inequality based on rigidity methods and introduce a related nonlinear flow.}\end{quote}\end{spacing}



\noindent{\bf 2010 MR Subject Classification:} 26D10; 46E35; 35K55; 58J60.

\thispagestyle{empty}

\section{Introduction}\label{Sec:Intro}
In this paper we consider the Moser-Trudinger-Onofri inequality, or \emph{Onofri inequality}, for brevity. This inequality takes any of the three following forms, which are all equivalent.
\begin{description}
\item[$\triangleright$] The Euclidean Onofri inequality:
\be{Onofri:Euclidean}
\frac1{16\,\pi}\irdeux{|\nabla u|^2}\ge\log\(\irdmu{e^u}\)-\irdmu u\,.
\ee
Here $d\mu=\mu(x)\,dx$ denotes the probability measure defined by $\mu(x)=\frac1\pi\,(1+|x|^2)^{-2}$, $x\in\R^2$.
\item[$\triangleright$] The Onofri inequality on the two-dimensional sphere $\S^2$:
\be{Onofri:Sphere}
\frac14\iS{|\nabla v|^2}\ge\log\(\iS{e^v}\)-\iS v\,.
\ee
Here $d\sigma$ denotes the uniform probability measure, that is, the measure induced by Lebesgue's measure on the unit sphere $\S^2\subset\R^3$ divided by a $4\pi$ factor.
\item[$\triangleright$] The Onofri inequality on the two-dimensional cylinder $\mathcal C=\S^1\times\R$:
\be{Onofri:Cylinder}
\frac1{16\,\pi}\iC{|\nabla w|^2}\ge\log\(\iC{e^w\,\nu}\)-\iC{w\,\nu}\,.
\ee
Here $y=(\theta,s)\in\mathcal C=\S^1\times\R$ and $\nu(y)=\frac1{4\pi}\,(\cosh s)^{-2}$ is a weight.
\end{description}
These three inequalities are equivalent. Indeed, on $\S^2\subset\R^3$, let us consider the coordinates $(\omega,z)\in\R^2\times\R$ such that $|\omega|^2+z^2=1$ and $z\in[-1,1]$. Let $\rho:=|\omega|$ and define the \emph{stereographic projection} $\Sigma:\S^2\setminus\{\mathrm N\}\to\R^2$ by $\Sigma(\omega)=x=r\,\omega/\rho$ and
\[
z=\frac{{}r^2-1}{r^2+1}=1-\frac{2}{r^2+1}\;,\quad \rho=\frac{{}2\,r}{r^2+1}\,.
\]
The \emph{North Pole} $\mathrm N$ corresponds to $z=1$ (and is formally sent at infinity) while the \emph{equator} (corresponding to $z=0$) is sent onto the unit sphere $\S^1\subset\R^2$. Whereas on the cylinder $\mathcal C$, we can consider the \emph{Emden-Fowler transformation} using the coordinates $\theta=x/|x|=\omega\rho$ and $s=-\log r=-\log|x|$. The functions $u$, $v$ and $w$ in \eqref{Onofri:Euclidean}, \eqref{Onofri:Sphere} and \eqref{Onofri:Cylinder} are then related by
\[
u(x)=v(\omega,z)=w(\theta,s)\,.
\]

\section{A review of the literature}\label{Sec:Review}

Inequality \eqref{Onofri:Sphere} has been established in \cite{MR0301504} without a sharp constant, based on the Moser-Trudinger inequality which was itself proved in \cite{Tru,MR0301504}, and in \cite{MR677001} with a sharp constant. For this reason it is sometimes called the \emph{Moser-Trudinger-Onofri inequality} in the literature. The result of E.~Onofri strongly relies on a paper of T. Aubin, \cite{MR534672}, which contains a number of results of existence for inequalities of Onofri type on manifolds (with unknown optimal constants). Also based on the Moser-Trudinger inequality, one has to mention \cite{MR960228} which connects Inequality~\eqref{Onofri:Sphere} with the Lebedev-Milin inequalities.

Concerning the other equivalent forms of the inequality, we may refer to \cite{MR2437030} for \eqref{Onofri:Cylinder} while it is more or less a standard result that \eqref{Onofri:Euclidean} is equivalent to \eqref{Onofri:Sphere}; an important result concerning this last point is the paper of E.~Carlen and M.~Loss, \cite{MR1143664}, which will be considered in more detail in Section~\ref{Sec:Duality}. Along the same line of thought, one has to mention \cite{MR1230930}, which also is based on the Funk-Hecke formula for the dual inequality, as was E.~Lieb's work on Hardy-Littlewood-Sobolev inequalities on the sphere, \cite{MR717827}.

The optimal function can be identified using the associated Euler-Lagrange equation, see \cite[Lemma~3.1]{MR845999} which provides details that are somewhat lacking in Onofri's original paper. We may also refer to~\cite[Theorem~12]{MR2509375} for multiplicity results for a slightly more general equation than the one associated with \eqref{Onofri:Euclidean}.

Another strategy can be found in the existing literature. In \cite{MR2154301}, A.~Ghigi provides a proof based on the Pr\'ekopa-Leindler inequality, which is also explained in full detail in the book \cite[Chapters 16-18]{MR3052352} of N.~Ghoussoub and A.~Moradifam. Let us mention that the book contains much more material and tackles the issues of improved inequalities under additional constraints, a question that was raised in \cite{MR534672} and later studied in \cite{MR925123,MR908146,MR2670931}.

Symmetrization, which allows to prove that optimality in \eqref{Onofri:Euclidean}, \eqref{Onofri:Sphere} or~\eqref{Onofri:Cylinder} is achieved among functions which are respectively radial (on the Euclidean space), or depend only on the azimuthal angle (the latitude, on the sphere), or only on the coordinate along the axis (of the cylinder) are an essential tool to reduce the complexity of the problem. For brevity, we shall refer to the \emph{symmetric case} when the function set is reduced to one of the above cases. Symmetrization results are widespread in the mathematical literature, so we shall only quote a few key papers. A standard reference is the paper of \cite{MR0402083} and in particular \cite[Theorem~2]{MR0402083} which is a key result for establishing the Hardy-Littlewood-Sobolev inequalities on the sphere and its limiting case, the logarithmic Hardy-Littlewood-Sobolev inequality. By duality and by considering the optimality case, one gets a symmetry result for the Onofri inequality, that can be found for instance in \cite{MR1143664}. It is also standard that the kinetic energy (Dirichlet integral) is decreased by symmetrization (a standard reference in the Euclidean case can be found in \cite[Lemma 7.17]{MR1817225}; also see \cite[p.~154]{MR929981}) and the adaptation to the sphere is straightforward. Historically, this was known much earlier and one can for instance quote \cite{MR0301504} (without any justification) and \cite[Lemma~1 and~2, p.~586]{MR0448404}. This is precisely stated in the context of the Onofri inequality on $\S^2$ in \cite[Lemma~1]{MR2154301}, which itself refers to \cite[Corollary~3 p.~60]{MR1297773} and \cite{MR810619}. A detailed statement can be found in \cite[Lemma 17.1.2]{MR3052352}. \emph{Competing symmetries} is another aspect of symmetrization that we will not study in this paper and for which we refer to \cite{MR1143664}.

In \cite{MR2473271}, Y.A.~Rubinstein gives a proof of the Onofri inequality that does not use symmetrization/rearrangement arguments. Also see \cite{MR2419932} and in particular \cite[Corollary 10.12]{MR2419932} which contains a reinforced version of the inequality. In \cite[Remark~(1), page 217]{MR908146}, there is another proof which does not rely on symmetry, based on a result in \cite{MR0292357}. Another proof that went rather unnoticed is used in the paper of E.~Fontenas \cite{MR1435336}. This approach is based on the so-called $\Gamma_2$ or \emph{carr\'e du champ} method. In the symmetric case the problem can be reduced to an inequality involving the ultraspherical operator that we will consider in Section~\ref{Sec:Rigidity}: see \eqref{Ultraspherical}, with $\lambda=1$. As far as we know, the first observation concerning this equivalent formulation can be found in \cite{MR1231419}, although no justification of the symmetrization appears in this paper. In a series of recent papers, \cite{DEKL,DEKL2012,1301,1307,Dolbeault2013437,1302} two of the authors have clarified the link that connects the \emph{carr\'e du champ} method with rigidity results that can be found in \cite{MR1134481} and earlier papers. Even better, their method involves a nonlinear flow which produces a remainder term, which will be considered in Section~\ref{Sec:Flow}.

Spherical harmonics play a crucial but hidden role, so we shall not insist on them and refer to \cite{MR1230930} and, in the symmetric case, to \cite[Chapter~16]{MR3052352} for further details. As quoted in \cite{MR3052352}, other variations on the Onofri-Moser-Trudinger inequality were given in~\cite{MR1646323,MR878016,MR1171306,MR981664,MR925123,MR908146}. The question of dimensions higher than $d=2$ is an entire topic by itself and one can refer to \cite{MR1230930,branson712moser,MR2377499,MR3089736} for some results in this direction. Various extensions of the Moser-Trudinger and Moser-Trudinger-Onofri inequalities have been proposed, which are out of the scope of this paper; let us simply mention \cite{MR3053467} as a contribution in this direction and refer the interested reader to references therein.

In this paper, we will neither cover issues related to conformal invariance, that were central in \cite{MR677001}, nor motivations arising from differential geometry. The reader interested in understanding how Onofri's inequality is related to the problem of prescribing the Gaussian curvature on $\S^2$ is invited to refer to \cite[Section~3]{MR934274} for an introductory survey, and to \cite{MR925123,MR908146,MR1989228} for more details.

Onofri's inequality also has important applications, for instance in chemotaxis: see \cite{MR1654677,MR2433703} in the case of the Keller-Segel model.

\medskip As a conclusion of this review, we can list the main tools that we have found in the literature:
\begin{enumerate}
\item[(T1)] Existence by variational methods,
\item[(T2)] Symmetrization techniques which allow to reduce the problem for \eqref{Onofri:Euclidean} to radial functions,
\item[(T3)] Identification of the solutions to the Euler-Lagrange equations (among radially symmetric functions),
\item[(T4)] Duality with the logarithmic Hardy-Littlewood-Sobolev inequality and study of the logarithmic Hardy-Littlewood-Sobolev inequality based on spherical harmonics and the Funk-Hecke formula,
\item[(T5)] Convexity methods related with the Pr\'ekopa-Leindler inequality,
\item[(T6)] $\Gamma_2$ or \emph{carr\'e du champ} methods,
\item[(T7)] Limiting procedures based on other functional inequalities.
\end{enumerate}
With these tools we may try to summarize the strategies of proof that have been developed. The approach of E.~Onofri is based on (T1)+(T2)+(T3), while (T4), (T5), (T6) and (T7) have been used in four independent and alternative strategies of proofs. None of them is elementary, in the sense that they rely on fundamental, deep or rather technical intermediate results.

\medskip In this paper, we intend to give new methods which, although not being elementary, are slightly simpler, or open new lines of thought. They also provide various additional terms which are all improvements. Several of them are based on the use of nonlinear flows, which, as far as we know, have not been really considered up to now, except in \cite{MR2915466,dolbeault:hal-00915998}. They borrow some key issues from at least one of the above mentioned tools (T1-7) or enlarge the framework.
\begin{enumerate}
\item \emph{Limiting procedures} based on other functional inequalities than Onofri's one, as in (T7), will be considered in Section~\ref{Sec:Limiting}. Six cases are studied, none of them being entirely new, but we thought that it was quite interesting to collect them. They also justify why we claim that ``the Onofri inequality plays in dimension two a role similar to the Sobolev inequality in higher dimensions.'' Other preliminary results (linearization, and (T2): symmetry results) will be considered in Sections~\ref{Sec:Linearization} and \ref{Sec:Symmetrization}.
\item Section~\ref{Sec:MassTransportation} is devoted to a \emph{mass transportation} approach of Onofri's inequality. Because of (T5), it was to be expected that such a technique would apply, at least formally (see Section~\ref{Sec:MassFormal}). A rigorous proof is established in the symmetric case in Section~\ref{Sec:MassRadialCase} and the consistency with a mass transportation approach of Sobolev's inequalities is shown in Section~\ref{Sec:MassApprox}. We have not found any result in this direction in the existing literature. (T2) is needed for a rigorous proof.
\item In Section~\ref{Sec:Duality}, we will come back to \emph{duality methods}, and get a first improvement to the standard Onofri inequality based on a simple \emph{expansion of a square}. This has of course to do with (T4) and (T5) but Proposition~\ref{Prop:Duality} is, as far as we know, a new result. We also introduce the super-fast (or logarithmic) diffusion, which has striking properties in relation with Onofri's inequality and duality, but we have not been able to obtain an improvement of the inequality as it has been done in the case of Sobolev's inequalities in \cite{dolbeault:hal-00915998}.
\item In Section~\ref{Sec:Fast-Diffusion}, we observe that in dimension $d=2$, the Onofri inequality is the natural functional inequality associated with the \emph{entropy--entropy production method for the fast diffusion equation with exponent $m=1/2$}. It is remarkable that no singular limit has to be taken. Moreover, the entropy--entropy production method provides an integral remainder term which is new.
\item In the last section (Section~\ref{Sec:Rigidity}), we establish \emph{rigidity results}. Existence of optimal functions is granted by (T1). Our results are equivalent to whose obtained with $\Gamma_2$ or \emph{carr\'e du champ} methods. This had already been noticed in the literature (but the equivalence of the two methods has never been really emphasized as it should have been). For the sake of simplicity, we start by a proof in the symmetric case in Section~\ref{Sec:RigiditySym}. However, our method does not \emph{a priori} require (T2) and directly provides essential properties for (T3), that is, the uniqueness of the solutions up to conformal invariance (for the critical value of a parameter, which corresponds to the first bifurcation point from the branch of the trivial constant solutions). Not only this point is remarkable, but we are also able to exhibit a nonlinear flow (in Section~\ref{Sec:Flow}) which unifies the various approaches and provides a new integral remainder term. Our main results in this perspective are collected in Section~\ref{Sec:General}.
\end{enumerate}

\section{Preliminaries}\label{Sec:Preliminaries}

\subsection{Onofri's inequality as a limit of various interpolation inequalities}\label{Sec:Limiting}

Onofri's inequality appears as an endpoint of various families of interpolation inequalities and corresponds to a critical case in dimension $d=2$ exactly like Sobolev's inequality when $d\ge3$. This is why one can claim that it plays in dimension two a role similar to the Sobolev inequality in higher dimensions. Let us give some six examples of such limits, which are probably the easiest way of proving Onofri's inequality.

\subsubsection{Onofri's inequality as a limit of interpolation inequalities on~$\S^2$}
On the sphere $\S^2$, one can derive the Onofri inequality from a family of interpolation inequalities on $\S^2$. We start from
\be{Ineq:Beckner}
\frac{q-2}2\,\nrmS{\nabla f}2^2+\nrmS f2^2\ge\nrmS fq^2\,,
\ee
which holds for any $f\in\H^1(\S^2)$. See \cite{MR1134481,MR1230930,DEKL2012}. Proceeding as in \cite{MR1230930} (also see \cite{MR2437030}), we choose $q=2\,(1+t)$, $f=1+\frac1{2\,t}\,v$, for any positive $t$ and use~\eqref{Ineq:Beckner}. This gives
\begin{multline*}
\(\frac1{4\,t}\iS{\left|\nabla v\right|^2}+1+\frac1t\iS v+\frac1{4\,t^2}\iS{\left|v\right|^2}\)^{1+t}\\
\ge\iS{\left|1+\frac1{2\,t}\,v\right|^{2\,(1+t)}}\,.
\end{multline*}
By taking the limit $t \to \infty$, we recover \eqref{Onofri:Sphere}.

\subsubsection{Onofri's inequality as a limit of Gagliardo-Nirenberg inequalities}

Consider the following sub-family of Gagliardo-Nirenberg inequalities
\be{Ineq:GN}
\nr f{2p}\le\mathsf C_{p,d}\,\nr{\nabla f}2^\theta\,\nr f{p+1}^{1-\theta}\,,
\ee
with $\theta=\theta(p):=\frac{p-1}p\,\frac d{d+2-p\,(d-2)}$, $1<p\le\frac d{d-2}\;\mbox{if}\;d\ge3$ and $1<p<\infty\;\mbox{if}\;d=2$. Such an inequality holds for any smooth function $f$ with sufficient decay at infinity and, by density, for any function $f\in \L^{p+1}(\R^d)$ such that $\nabla f$ is square integrable. We shall assume that $\mathsf C_{p,d}$ is the best possible constant. In \cite{MR1940370}, it has been established that equality holds in \eqref{Ineq:GN} if $f=F_p$ with
\be{Eqn:Optimal}
F_p(x)=(1+|x|^2)^{-\frac1{p-1}}\quad\forall\;x\in\R^d\,,
\ee
and that all extremal functions are equal to $F_p$ up to multiplication by a constant, a translation and a scaling. If $d\ge 3$, the limit case $p=d/(d-2)$ corresponds to Sobolev's inequality and one recovers the results of T.~Aubin and G.~Talenti in~\cite{MR0448404,MR0463908}, with $\theta=1$: the optimal functions for it are, up to scalings, translations and multiplications by a constant, all equal to $F_{d/(d-2)}(x)=(1+|x|^2)^{-(d-2)/2}$, and
\[
\mathsf S_d=(\mathsf C_{d/(d-2),\kern 1pt d})^2\,.
\]
We can recover the Euclidean Onofri inequality as the limit case $d=2$, $p\to\infty$ in the above family of inequalities, in the following way:
\begin{Prop}\label{Prop:LimitGN}{\rm \cite{MR2915466}} Assume that $u\in\mathcal D(\R^2)$ is such that $\irdmu u=0$ and let
\[
f_p:=F_p\,\(1+\frac u{2\,p}\)\;,
\]
where $F_p$ is defined by~\eqref{Eqn:Optimal}. Then we have
\[
1\le\lim_{p\to\infty}\mathsf C_{p,2}\,\frac{\nrm{\nabla f_p}2^{\theta(p)}\,\nrm{f_p}{p+1}^{1-\theta(p)}}{\nrm{f_p}{2p}}=\frac{e^{\frac1{16\,\pi}\,\irdeux{|\nabla u|^2}}}{\irdmu{e^{\,u}}}\;.
\]
\end{Prop}
\noindent We recall that $\mu(x):=\frac1\pi\,(1+|x|^2)^{-2}$, and $d\mu(x):=\mu(x)\,dx$.
\begin{proof} For completeness, let us give a short proof. We can rewrite~\eqref{Ineq:GN} as
\[
\frac{\irdeux{|f|^{2p}}}{\irdeux{|F_p|^{2p}}}\le\(\frac{\irdeux{|\nabla f|^2}}{\irdeux{|\nabla F_p|^2}}\)^\frac{p-1}2\,\frac{\irdeux{|f|^{p+1}}}{\irdeux{|F_p|^{p+1}}}
\]
and observe that, with $f=f_p$, we have:\\
(i) $\lim_{p\to\infty}\irdeux{|F_p|^{2p}}=\irdeux{\frac1{(1+|x|^2)^2}}=\pi$ and
\[
\lim_{p\to\infty}\irdeux{|f_p|^{2p}}=\irdeux{F_p^{2p}\,(1+\tfrac u{2p})^{2p}}=\irdeux{\frac{e^u}{(1+|x|^2)^2}}\,,
\]
so that $\irdeux{|f_p|^{2p}}/\irdeux{|F_p|^{2p}}$ converges to $\irdmu{e^{\,u}}$ as $p\to\infty$.\\
(ii) $\irdeux{|F_p|^{p+1}}=(p-1)\,\pi/2$, $\lim_{p\to\infty}\irdeux{|f_p|^{p+1}}=\infty$, but
\[
\lim_{p\to\infty}\frac{\irdeux{|f_p|^{p+1}}}{\irdeux{|F_p|^{p+1}}}=1\;.
\]
(iii) Expanding the square and integrating by parts, we find that
\begin{multline*}
\irdeux{|\nabla f_p|^2}=\frac1{4p^2}\irdeux{F_p^2\,|\nabla u|^2}-\irdeux{(1+\tfrac u{2p})^2\,F_p\,\Delta F_p}\\=\frac1{4p^2}\irdeux{|\nabla u|^2}+\frac{2\pi}{p+1}+o(p^{-2})\;.
\end{multline*}
Here we have used $\irdeux{|\nabla F_p|^2}=\frac{2\pi}{p+1}$ and the condition $\irdmu u=0$ in order to discard one additional term of the order of $p^{-2}$. On the other hand, we find that
\[
\(\frac{\irdeux{|\nabla f_p|^2}}{\irdeux{|\nabla F_p|^2}}\)^\frac{p-1}2\sim\(1+\frac{p+1}{8\,\pi\,p^2}\irdeux{|\nabla u|^2}\)^\frac{p-1}2\sim e^{\frac1{16\,\pi}\,\irdeux{|\nabla u|^2}}
\]
as $p\to\infty$. Collecting these estimates concludes the proof.\end{proof}

\subsubsection{Onofri's inequality as a limit of Sobolev inequalities}\label{Sec:LimitSobolev2d}

Another way to derive Onofri's inequality is to consider the usual optimal Sobolev inequalities in $\mathbb R^2$, written for an $\L^p(\R^2)$ norm of the gradient, for an arbitrary $p\in(1,2)$. This method is inspired by \cite{MR3089736}, which is devoted to inequalities in exponential form in dimensions $d\ge2$. See in particular \cite[Example~1.2]{MR3089736}. In the special case $p\in(1,2)$, $d=2$, let us consider the Sobolev inequality
\be{Ineq:Sobolev2}
\nrm f{\frac{2\,p}{2-p}}^p\leq\mathsf C_p\,\nrm{\nabla f}p^p\quad\forall\,f\in\mathcal D(\R^2)\,,
\ee
where equality is achieved by the Aubin-Talenti extremal profile
\[
f_\star(x)=\(1+|x|^{\frac p{p-1}}\)^{-\frac{2-p}p}\quad\forall\,x\in\R^2\,.
\]
The extremal functions were already known from the celebrated papers by T.~Aubin and G.~Talenti, \cite{MR0448404,MR0463908}. See also~\cite{bliss1930integral,MR0289739} for earlier related computations, which provided the value of some of the best constants. It is easy to check that $f_\star$ solves
\[
-\Delta_pf_\star=2\(\frac{2-p}{p-1}\)^{p-1}\,f_\star^{\frac{2\,p}{2-p}-1}\,,
\]
hence
\[
\nrm{\nabla f_\star}p^p=\frac1{\mathsf C_p}\,\nrm{f_\star}{\frac{2\,p}{2-p}}^p=2\(\frac{2-p}{p-1}\)^{p-1}\,\nrm{f_\star}{\frac{2\,p}{2-p}}^{\frac{2\,p}{2-p}}\;,
\]

so that the optimal constant is
\[
\mathsf C_p=\frac 12\(\frac{p-1}{2-p}\)^{p-1}\(\frac{p^2\,|\sin(2\,\pi/p)|}{2\,(p-1)\,(2-p)\,\pi^2}\)^{\!\frac p2}\,.
\]
We can study the limit $p\to2_-$ in order to recover the Onofri inequality by considering $f=f_\star\,\big(1+\frac{2-p}{2\,p}\,u\big)$, where $u$ is a given smooth, compactly supported function, and $\varepsilon=\frac{2-p}{2\,p}$. A direct computation gives
\[
\lim_{p\to2_-}\irdeux{f^\frac{2\,p}{2-p}}=\irdeux{\frac{e^u}{(1+|x|^2)^2}}=\pi\irdmu{\,e^u}\,,
\]
and
\begin{multline*}
\irdeux{|\nabla f|^p}=2\,\pi\,(2-p)\,\left[1+\tfrac{2-p}2\irdmu u\right]\\
+(\tfrac{2-p}{2\,p})^p\irdeux{|\nabla u|^2}+o((2-p)^2)\,,
\end{multline*}
as $p\to2_-$. By taking the logarithm of both sides of \eqref{Ineq:Sobolev2}, we get
\begin{multline*}
\frac{2-p}2\,\log\(\irdmu{\,e^u}\)\sim\frac{2-p}2\,\log\(\frac{\irdeux{f^\frac{2\,p}{2-p}}}{\irdeux{f_\star^\frac{2\,p}{2-p}}}\)\\
\le\log\(\frac{\irdeux{|\nabla f|^p}}{\irdeux{|\nabla f_\star|^p}}\)\hspace*{4cm}\\
=\log\(1+\tfrac{2-p}2\irdmu u+\tfrac{2-p}{32\,\pi}\irdeux{|\nabla u|^2}+o(2-p)\)
\end{multline*}
Gathering the terms of order $2-p$, we recover the Euclidean Onofri inequality by passing to the limit $p\to2_-$.

\subsubsection{The radial Onofri inequality as a limit when $d\to2$}

Although this approach is restricted to radially symmetric functions, one of the most striking way to justify the fact that the Onofri inequality plays in dimension two a role similar to the Sobolev inequality in higher dimensions goes as follows. To start with, one can consider the Sobolev inequality applied to radially symmetric functions only. The dimension $d$ can now be considered as a real parameter. Then, by taking the limit $d\to2$, one can recover a weaker (\emph{i.e.} for radial functions only)
version of the Onofri inequality. The details of the computation, taken from \cite{dolbeault:hal-00915998}, follow.

Consider the radial Sobolev inequality
\be{Ineq:SobolevRadial}
\mathsf{s}_d\int_0^\infty |f'|^2\,r^{d-1}\,dr\ge\(\int_0^\infty |f|^\frac{2\,d}{d-2}\,r^{d-1}\,dr\)^{1-\frac2d}\;,
\ee
with optimal constant
\[
\mathsf{s}_d=\frac4{d\,(d-2)}\(\frac{\Gamma\(\frac{d+1}2\)}{\sqrt\pi\,\Gamma\(\frac{d}2\)}\)^\frac2d\,.
\]
We may pass to the limit in~\eqref{Ineq:SobolevRadial} with the choice

\[
f(r)=f_\star(r)\(1+\tfrac{d-2}{2\,d}\,u\)\,,
\]
where $f_\star(r) = (1+r^2)^{-\frac{d-2}{2}}$ gives the equality case in \eqref{Ineq:SobolevRadial},
to get the radial version of Onofri's inequality for $u$. By expanding the expression of $|f'|^2$ we get
\[
f'^2=f_\star'^2+\frac{d-2}d\,f_\star'\(f_\star\,u\)'+\(\frac{d-2}{2\,d}\)^2\(f_\star'\,u+f_\star\,u'\)^2\,.
\]

We have
\[
\lim_{d\to2_+}\int_0^\infty |f_\star\,(1+\tfrac{d-2}{2\,d}\,u)|^\frac{2\,d}{d-2}\,r^{d-1}\,dr=\int_0^\infty e^u\;\frac{r\,dr}{(1+r^2)^2}\,,
\]
so that, as $d\to2_+$,
\[
\(\int_0^\infty |f_\star\,(1+\tfrac{d-2}{2\,d}\,u)|^\frac{2\,d}{d-2}\,r^{d-1}\,dr\)^\frac{d-2}d-1
\sim\frac{d-2}2\,\log\(\int_0^\infty e^u\;\frac{r\,dr}{(1+r^2)^2}\).
\]
Also, using the fact that
\[
\mathsf{s}_d=\frac1{d-2}+\frac{1}2-\frac{1}2\,\log 2+o(1)\quad\mbox{as}\quad d\to2_+\,,
\]
we have
\[
\mathsf s_d \int_{0}^{\infty} |f'|^2\, r^{d-1}\; dr \sim 1
+ (d-2) \left[ \frac18\int_0^\infty |u'|^2\,r\,dr+\int_0^\infty u\;\frac{2\,r\,dr}{(1+r^2)^2} \right]\,.
\]
By keeping only the highest order terms, which are of the order of $(d-2)$, and passing to the limit as $d\to2_+$ in~\eqref{Ineq:SobolevRadial}, we obtain that
\[
\frac18\int_0^\infty |u'|^2\,r\,dr+\int_0^\infty u\;\frac{2\,r\,dr}{(1+r^2)^2}\ge\log\(\int_0^\infty e^u\;\frac{2\,r\,dr}{(1+r^2)^2}\)\,,
\]
which is Onofri's inequality written for radial functions.

\subsubsection{Onofri's inequality as a limit of Caffarelli-Kohn-Nirenberg inequalities}

Onofri's inequality can be obtained as the limit in a familly of Caffarelli-Kohn-Nirenberg inequalities,
as was first done in \cite{MR2437030}.

Let $2^*:=\infty$ if $d=1$ or $2$, $2^*:=2\,d/(d-2)$ if $d\ge3$ and $a_c:=(d-2)/2$. Consider the space $\mathcal D_a^{1,2}(\R^d)$ obtained by completion of $\mathcal D(\R^d\setminus\{0\})$ with respect to the norm $u\mapsto\nr{\,|x|^{-a}\,\nabla u\,}2^2$. In this section, we shall consider the Caffarelli-Kohn-Nirenberg inequalities
\be{Ineq:CKN}
\(\;\ird{\frac{|u|^p}{|x|^{bp}}}\)^{\frac2p}\le\mathsf{C}_{a,b}\ird{\frac{|\nabla u|^2}{|x|^{2a}}}\;.
\ee
These inequalities generalize to $\mathcal D_a^{1,2}(\R^d)$ the Sobolev inequality and in particular the exponent $p$ is given in terms of $a$ and $b$ by
\[
p=\frac{2\,d}{d-2+2\,(b-a)}\;,
\]
as can be checked by a simple scaling argument. A precise statement on the range of validity of~\eqref{Ineq:CKN} goes as follows.
\begin{Lem}\label{Lem:CKN}{\rm \cite{MR768824}} Let $d\ge 1$. For any $p\in [2, 2^*]$ if $d\ge3$ or $p\in [2, 2^*)$ if $d=1$ or $2$, there exists a positive constant $\mathsf{C}_{a,b}$ such that \eqref{Ineq:CKN} holds if $a$, $b$ and $p$ are related by $b=a-a_c+d/p$, with the restrictions $a<a_c$, $a\le b\le a+1$ if $d\ge3$, $a<b\le a+1$ if $d=2$ and $a+1/2<b\le a+1$ if $d=1$.\end{Lem}

We shall restrict our purpose to the case of dimension $d=2$. For any $\alpha\in(-1,0)$, let us denote by $d\mu_\alpha$ the probability measure on $\R^2$ defined by $d\mu_\alpha:=\mu_\alpha\,dx$ where
\[
\mu_\alpha:=\frac{1+\alpha}\pi\,\frac{|x|^{2\,\alpha}}{(1+|x|^{2\,(1+\alpha)})^2}\,.
\]
It has been established in~\cite{MR2437030} that
\be{MTalpha}
\log\(\int_{\R^2}e^{\,u}\;d\mu_\alpha\)-\int_{\R^2} u\;d\mu_\alpha\le\frac{1}{16\,\pi\,(1+\alpha)}\,\irdeux{|\nabla u|^2}\quad\forall\;u\in\mathcal D(\R^2)\,,
\ee
where $\mathcal D(\R^2)$ is the space of smooth functions with compact support. By density with respect to the natural norm defined by each of the inequalities, the result also holds on the corresponding Orlicz space.

We adopt the strategy of \cite[Section~2.3]{MR2437030} to pass to the limit in~\eqref{Ineq:CKN} as $(a,b)\to(0,0)$ with $b=\frac{\alpha}{\alpha + 1}\,a$. Let
\[
a_\varepsilon=-\frac{\varepsilon}{1-\varepsilon}\,(\alpha+1)\,,\quad b_\varepsilon=a_\varepsilon+\varepsilon ,\quad p_\varepsilon=\frac{2}\varepsilon\,,
\]
and
\[
u_\varepsilon (x) = \(1+|x|^{2\,(\alpha+1)}\)^{-\frac\varepsilon{1-\varepsilon}}\,.
\]
Assuming that $u_\varepsilon$ is an optimal function for \eqref{Ineq:CKN}, define
\[
\kappa_\varepsilon =\int_{\R^2} \left[ \frac{u_\varepsilon}{|x|^{a_\varepsilon + \varepsilon}} \right]^{2/\varepsilon}\, dx = \int_{\R^2}\frac{|x|^{2\,\alpha}}{\big(1+|x|^{2\,(1+\alpha)}\big)^2}\,\frac{u_\varepsilon^2}{|x|^{2a_\varepsilon}}\,dx =\frac\pi{\alpha+1}\,\frac{\Gamma\big(\frac1{1-\varepsilon}\big)^2}{\Gamma\big(\frac2{1-\varepsilon}\big)}\,,
\]
\[
\lambda_\varepsilon = \int_{\R^2}\left[ \frac{|\nabla u_\varepsilon|}{|x|^a} \right]^2\,dx = 4\,a_\varepsilon^2\int_{\R^2}\frac{|x|^{2\,(2\,\alpha+1-a_\varepsilon)}}{\big(1+|x|^{2\,(1+\alpha)}\big)^{\frac2{1-\varepsilon}}}\,dx= 4\,\pi\,\frac{|a_\varepsilon|}{1-\varepsilon}\,\frac{\Gamma\big(\frac{1}{1-\varepsilon}\big)^2}{\Gamma\big(\frac{2}{1-\varepsilon}\big)}\,.
\]
Then $w_\varepsilon=(1+\frac{1}2\,\varepsilon\,u)\,u_\varepsilon$ is such that
\begin{eqnarray*}
&&\lim_{\varepsilon\to0_+}\frac{1}{\kappa_\varepsilon}\int_{\R^2}\frac{|w_\varepsilon|^{p_\varepsilon}}{|x|^{b_\varepsilon{p_\varepsilon}}}\,dx=\int_{\R^2}e^u\,d\mu_\alpha\,,\\
&&\lim_{\varepsilon\to0_+}\frac{1}\varepsilon\left[\frac{1}{\lambda_\varepsilon}\int_{\R^2}\frac{|\nabla w_\varepsilon|^2}{|x|^{2a_\varepsilon}}\,dx-1\right]=\int_{\R^2}u\,d\mu_\alpha+\frac1{16\,(1+\alpha)\,\pi}\,\nrm{\nabla u}2^2\,.
\end{eqnarray*}

\subsubsection{Limits of some Gagliardo-Nirenberg inequalities on the line}

Onofri's inequality on the cylinder, \eqref{Onofri:Cylinder} can also be recovered by a limiting process, in the symmetric case. As far as we know, this method for proving the inequality is new, but a clever use of the Emden-Fowler transformation and of the results based on the Caffarelli-Kohn-Nirenberg inequalities shows that this was to be expected. See \cite{MR2437030} for more considerations in this direction.

Consider the Gagliardo-Nirenberg inequalities on the line
\[
\nrml fp\le\mathsf C_{\mathrm{GN}}^p\,\nrml{f'}2^\theta\,\nrml f2^{1-\theta}\quad\forall\,f\in\mathrm H^1(\R)\;,
\]
with $\theta=\frac{p-2}{2\,p}$, $p>2$. Equality is achieved by the function
\[
f_\star(x):=\(\cosh s\)^{-\frac2{p-1}}\quad\forall\,s\in\R\,.
\]
See \cite{1307} for details. By taking the logarithm of both sides of the inequality, we find that
\[
\frac2p\,\log\(\frac{\irl{f^p}}{\irl{f_\star^p}}\)\le\theta\,\log\(\frac{\irl{|f'|^2}}{\irl{|f_\star'|^2}}\)+(1-\theta)\,\log\(\frac{\irl{f^2}}{\irl{f_\star^2}}\)
\]
and elementary computations show that as $p\to+\infty$, $f_\star^p\to2\,\xi$ and $-f_\star\,f_\star''\sim\frac4p\,\xi$ with $\xi(s):=\frac12\,\(\cosh s\)^{-2}$. If we take $f=f_\star\,(1+w/p)$, we have
\[
\lim_{p\to\infty}\irl{f^p}=2\irl{e^w\,\xi}\,,
\]
\[
\lim_{p\to\infty}\log\(\frac{\irl{f^p}}{\irl{f_\star^p}}\)=\log\(\irl{e^w\,\xi}\)\,.
\]
We can also compute
\[
\irl{|f_\star|^2}=\frac{p-1}2+2\,\log2+O\big(\tfrac1p\big)
\]
and
\[
\irl{|f|^2}=\irl{|f_\star|^2\,(1+\tfrac2p\,w+\tfrac1{p^2}\,w^2)}=\frac{p-1}2+2\,\log2+O\big(\tfrac1p\big)
\]
as $p\to+\infty$, so that
\[
\frac{\irl{f^2}}{\irl{f_\star^2}}-1=O\big(\tfrac1{p^2}\big)\quad\mbox{and}\quad\lim_{p\to\infty}p\,\log\(\frac{\irl{f^2}}{\irl{f_\star^2}}\)=0\,.
\]
For the last term, we observe that, pointwise,
\[
-f_\star\,f_\star''\sim\frac2p\,\frac1{(\cosh s)^2}
\]
and
\[
\irl{|f_\star'|^2}=-\irl{f_\star\,f_\star''}=\frac2p+O\big(\tfrac1{p^2}\big)\quad\mbox{as}\quad p\to+\infty\,.
\]
Passing to the limit as $p\to+\infty$, we get that
\begin{eqnarray*}
\irl{|f'|^2}&=&\frac1{p^2}\irl{f_\star^2\,|w'|^2}-\irl{f_\star\,f_\star''\,\(1+\frac wp\)^2}\\
&=&\frac1{p^2}\irl{|w'|^2}+\frac2p\(1+\frac4p\irl{w\,\xi}\)+o\big(\tfrac1{p^2}\big)\,,
\end{eqnarray*}
and finally
\[
\log\(\frac{\irl{|f'|^2}}{\irl{|f_\star'|^2}}\)\sim\frac1p\(4\irl{w\,\xi}+\frac12\irl{|w'|^2}\)+o\big(\tfrac1p\big)\,.
\]
Collecting terms, we find that
\[
\frac18\irl{|w'|^2}\ge\log\(\irl{e^w\,\xi}\)-\irl{w\,\xi}\,.
\]

\subsection{Linearization and optimal constant}\label{Sec:Linearization}

Consider \eqref{Onofri:Sphere} and define
\[
\mathcal I_\lambda:=\inf_{\begin{array}{c}v\in\H^1(\S^2)\cr\iS v>0\end{array}}\mathcal Q_\lambda[v]\quad\mbox{with}\quad\mathcal Q_\lambda[v]:=\frac{\frac14\,\iS{|\nabla v|^2}+\lambda\iS v}{\log\(\iS{e^v}\)}\,.
\]
By Jensen's inequality, $\log\(\iS{e^v}\)\ge\iS v>0$, so that $\mathcal I_\lambda$ is well defined and nonnegative for any $\lambda>0$. Since constant functions are admissible, we also know that
\[
\mathcal I_\lambda\le\lambda\;,
\]
for any $\lambda>0$. Moreover, since $\lambda\mapsto\mathcal Q_\lambda[v]$ is affine, we know that $\lambda\mapsto\mathcal I_\lambda$ is concave and continuous. Assume now that $\iS v=0$ and for any $c>0$, let us consider
\be{atlanta1}
\mathcal Q_\lambda[v+c]=\frac{\frac14\,\iS{|\nabla v|^2}+\lambda\,c}{c+\log\(\iS{e^v}\)}\ge\frac{\log\(\iS{e^v}\)+\lambda\,c}{c+\log\(\iS{e^v}\)}\;,
\ee
where the inequality follows from \eqref{Onofri:Sphere}. It is clear that for such functions $v$,
\begin{eqnarray*}
&&\lim_{c\to+\infty}\mathcal Q_\lambda[v+c]=\lambda\,,\\
&&\lim_{c\to0_+}\mathcal Q_\lambda[v+c]=\frac{\iS{|\nabla v|^2}}{\log\(\iS{e^v}\)}=\mathcal Q_\lambda[v]\,.
\end{eqnarray*}
If $\lambda<1$, using \eqref{atlanta1}, we can write that for all $c>0$,
\[
\mathcal Q_\lambda[v+c]\ge\lambda+(1-\lambda)\,\frac{\log\(\iS{e^v}\)}{c+\log\(\iS{e^v}\)}\ge\lambda\,,
\]
thus proving that $\mathcal I_\lambda=\lambda$ is optimal when $\lambda<1$.

When $\lambda\ge1$, we may take $v=\varepsilon\,\phi$, where $\phi$ is an eigenfunction of the Laplace-Beltrami operator $-\,\Delta_{\S^2}$ on the sphere $\S^2$, such that $-\,\Delta_{\S^2}\phi=2\,\phi$ and take the limit as $\varepsilon\to0_+$, so that $\iS{|\nabla v|^2}=\varepsilon^2\iS{|\nabla\phi|^2}=2\,\varepsilon^2\iS{|\phi|^2}$ and $\log\(\iS{e^v}\)=\log\(1+\frac12\,\varepsilon^2\iS{\phi^2}+o(\varepsilon^2)\)$. Collecting terms, we get that
\[
\lim_{\varepsilon\to0_+}\mathcal Q_\lambda[\varepsilon\,\phi]=1\,.
\]
Altogether, we have found that
\[
\mathcal I_\lambda=\min\{\lambda,1\}\quad\forall\,\lambda>0\,.
\]

\subsection{Symmetrization results}\label{Sec:Symmetrization}

For the sake of completeness, let us state a result of symmetry. Consider the functional
\[
\mathcal G_\lambda[v]:=\frac14\iS{|\nabla v|^2}+\lambda\,\iS{v}-\log\(\iS{e^v}\)\,,
\]
and denote by $\H^1_*(\S^2)$ the function in $\H^1(\S^2)$ which depend only on the azimuthal angle (latitude), that we shall denote by $\theta\in[0,\pi]$.
\begin{Prop}\label{Prop:Symmetry} For any $\lambda>0$,
\[
\inf_{v\in\H^1(\S^2)}\mathcal G_\lambda[v]=\inf_{v\in\H^1_*(\S^2)}\mathcal G_\lambda[v]\,.
\]
\end{Prop}
We refer to \cite[Lemma~17.1.2]{MR3052352} for a proof of the symmetry result and to Section~\ref{Sec:Review} for further historical references.

Hence, for any function $v\in\H^1(\S^2)$, the inequality $\mathcal G_1[v]\ge0$ reads
\[
\frac18\int_0^\pi|v'(\theta)|^2\,\sin\theta\,d\theta+\frac12\int_0^\pi v(\theta)\,\sin\theta\,d\theta\ge\log\(\frac12\int_0^\pi e^v\,\sin\theta\,d\theta\)\,.
\]
The change of variables $z=\cos\theta$, $v(\theta)=f(z)$ allows to rewrite this inequality as
\be{Ultraspherical}
\frac18\int_{-1}^1|f'|^2\,\nu\,dz+\frac12\int_{-1}^1f\,dz\ge\log\(\frac12\int_{-1}^1e^f\,dz\)\,,
\ee
where $\nu(x):=1-z^2$. Let us define the \emph{ultraspherical} operator $\LL$ by $\scal{f_1}{\LL f_2}=-\int_{-1}^1 f_1'\,f_2'\,\nu\,dz$ where $\scal\cdot\cdot$ denotes the standard scalar product on $\L^2(-1,1;dz)$. Explicitly we have:
\[
\LL f:=(1-z^2)\,f''-2\,z\,f'=\nu\,f''+\nu'\,f'
\]
and \eqref{Ultraspherical} simply means
\[
-\frac18\scal{f}{\LL f}+\frac12\int_{-1}^1f\,\nu\,dz\ge\log\(\frac12\int_{-1}^1e^f\,\nu\,dz\)\,.
\]

\section{Mass Transportation}\label{Sec:MassTransportation}

Since Onofri's inequality appears as a limit case of Sobolev's inequalities which can be proved by mass transportation according to \cite{MR2032031}, it makes a lot of sense to look for a proof based on such techniques. Let us start by recalling some known results.

Assume that $F$ and $G$ are two probability distributions on $\R^2$ and consider the convex function $\phi$ such that $G$ is the \emph{push-forward} of $F$ through~$\nabla\phi$
\[
\nabla\phi_*F=G\,,
\]
where $\nabla\phi$ is the \emph{Brenier map} and $\phi$ solves the Monge-Amp\`ere equation
\be{Eqn:MA}
F=G(\nabla\phi)\,\mathrm{det}\(\mathrm{Hess}(\phi)\)\quad\mbox{on}\;\R^d\,.
\ee
See \cite{MR1369395} for details. Here $d=2$ but to emphasize the role of the dimension, we will keep it as a parameter for a while. The Monge-Amp\`ere equation~\eqref{Eqn:MA} holds in the $F\,dx$ sense almost everywhere according to \cite[Remark~4.5]{MR1451422}, as discussed in \cite{MR2032031}. By now this strategy is standard and we shall refer to \cite{MR2459454} for background material and technical issues that will be omitted here. We can for instance assume that $F$ and $G$ are smooth functions and argue by density afterwards.

\subsection{Formal approach}\label{Sec:MassFormal}

Let us start by a formal computation. Using \eqref{Eqn:MA}, since
\[
G(\nabla\phi)^{-\frac1d}=F^{-\frac1d}\,\mathrm{det}\(\mathrm{Hess}(\phi)\)^\frac1d\le\frac1d\,F^{-\frac1d}\,\Delta\phi
\]
by the arithmetic-geometric inequality, we get the estimate
\be{Ineq:CNV}
\int_{\R^d}G(y)^{1-\frac1d}\,dy=\int_{\R^d}G(\nabla\phi)^{1-\frac1d}\,\mathrm{det}\(\mathrm{Hess}(\phi)\)\,dx\le\frac1d\int_{\R^d}F^{1-\frac1d}(x)\,\Delta\phi\,dx
\ee
using the change of variables $y=\nabla\phi(x)$ and \eqref{Eqn:MA}. Assume that
\[
G(y)=\mu(y)=\frac1{\pi\,(1+|y|^2)^2}\quad\forall\,y\in\R^d
\]
and
\[
F=\mu\,e^u\,.
\]
With $d=2$, we obtain
\begin{multline*}
4\irdeux{\sqrt\mu\,}=2\,d\int_{\R^d}G(y)^{1-\frac1d}\,dy\\
=2\int_{\R^d}F^{1-\frac1d}(x)\,\Delta\phi\,dx=-\irdeux{\nabla\log F\cdot\sqrt F\,\nabla\phi}\\
=-\irdeux{(\nabla\log\mu+\nabla u)\cdot\sqrt F\,\nabla\phi}\,,
\end{multline*}
which can be estimated using the Cauchy-Schwarz inequality by
\begin{multline*}
16\(\irdeux{\sqrt\mu\,}\)^2=\(\irdeux{(\nabla\log\mu+\nabla u)\cdot\sqrt F\,\nabla\phi}\)^2\\
\le\irdeux{|\nabla u+\nabla\log\mu|^2}\irdeux{F\,|\nabla\phi|^2}\,.
\end{multline*}
If we expand the square, that is, if we write
\begin{multline*}
\irdeux{|\nabla u+\nabla\log\mu|^2}\\
=\irdeux{|\nabla u|^2}-2\irdeux{u\,\Delta\log\mu}+\irdeux{|\nabla\log\mu|^2}\,,
\end{multline*}
after recalling that
\[
-\Delta\log\mu=8\,\pi\,\mu\;,
\]
and after undoing the change of variables $y=\nabla\phi(x)$, so that we get
\[
\irdeux{F\,|\nabla\phi|^2}=\irdeux{G(y)\,|y|^2}=\irdeux{\mu\,|y|^2}\,,
\]
we end up, after collecting the terms, with
\[
\frac{16\(\irdeux{\sqrt\mu\,}\)^2}{\irdeux{\mu\,|y|^2}}-\irdeux{|\nabla\log\mu|^2}\le\irdeux{|\nabla u|^2}+16\,\pi\irdmu u\,.
\]
Still at a formal level, we may observe that
\begin{multline*}
16\(\irdeux{\sqrt\mu\,}\)^2=\(-\,2\irdeux{y\cdot\nabla\sqrt\mu\,}\)^2\\
=\(\irdeux{y\,\sqrt\mu\cdot\nabla\log\mu\,}\)^2\\
=\irdeux{\mu\,|y|^2}\irdeux{|\nabla\log\mu|^2}
\end{multline*}
as it can easily be checked that~$y\sqrt{\mu}$ and~$\nabla\log\mu$ are proportional. This would prove Onofri's inequality since $\log\(\irdmu{e^u}\)=\log\(\irdeux F\)=0$, if $y\mapsto\sqrt\mu$, $y\mapsto\mu\,|y|^2$ and $y\mapsto|\nabla\log\mu|^2$ were integrable, but this is not the case. As we shall see in the next section, this issue can be solved by working on balls.

\subsection{The radially symmetric case}\label{Sec:MassRadialCase}

When $F$ and $G$ are assumed to depend only on $r=|x|$, so that we may write that $|y|=s=\varphi(r)$, then \eqref{Eqn:MA} becomes
\[
(G\circ\varphi')\,\(\frac{\varphi'}r\)^{d-1}\!\varphi''=F
\]
what allows to compute $\varphi'$ using
\[
\int_0^{\varphi'(R)}G(s)\;s^{d-1}\,ds=\int_0^R(G\circ\varphi')\,\(\frac{\varphi'}r\)^{d-1}\!\varphi''\;r^{d-1}\,dr=\int_0^RF(r)\;r^{d-1}\,dr\,.
\]
With a straightforward abuse of notation we shall indifferently write that $F$ is a function of $x$ or of $r$ and $G$ a function of $y$ or $s$.

The proof is similar to the one in Section~\ref{Sec:MassFormal} except that all integrals can be restricted to a ball $B_R$ of radius $R>0$ with center at the origin. Assume that $G=\mu/Z_R$, $F=e^u\,\mu/Z_R$ where $Z_R=\int_{B_R}\mu\,dx$ and $u$ has compact support inside the ball $B_R$. An easy computation shows that
\[
Z_R=\frac{R^2}{1+R^2}\quad\forall\,R>0\,.
\]
We shall also assume that $u$ is normalized so that $\int_{B_R}F\,dx=1$.

All computations are now done on $B_R$. The only differences with Section~\ref{Sec:MassFormal} arise from the integrations by parts, so we have to handle two additional terms:
\begin{multline*}
\int_{B_R}F^{1-\frac1d}(x)\,\Delta\phi\,dx+\frac12\int_{B_R}\nabla\log F\cdot\sqrt F\,\nabla\phi\,dx\\
=\pi\,R\,\sqrt{F(R)}\,\varphi'(R)=\pi\,R\,\sqrt{\mu(R)/Z_R}\,\varphi'(R)
\end{multline*}
and
\[
2\int_{B_R}\nabla u\cdot\nabla\log\mu\,dx+2\int_{B_R}u\,\Delta\log\mu\,dx=4\,\pi\,R\,(\log\mu)'(R)\,u(R)=0\,.
\]
If we fix $u$ (smooth, with compact support) and let $R\to\infty$, then it is clear that none of these two terms plays a role. Notice that there exists a constant $\kappa$ such that
\[
\frac{(\varphi'(R))^2}{1+(\varphi'(R))^2}=\frac{R^2}{1+R^2}+\kappa
\]
for large values of $R$, and hence $\varphi'(R)\sim R$. Hence,
\[
\lim_{R\to\infty}\pi\,R\,\sqrt{\mu(R)/Z_R}\,\varphi'(R)=\sqrt\pi\,.
\]
After collecting the terms, we obtain
\begin{multline*}
\frac{16\,\(\int_{B_R}\sqrt\mu\;dy-\sqrt\pi\)^2}{\int_{B_R}\mu\,|y|^2\,dy}-\int_{B_R}|\nabla\log\mu|^2\,dy+o(1)\\
\le\irdeux{|\nabla u|^2}-16\,\pi\irdmu u
\end{multline*}
as $R\to\infty$. Using the equality case for the Cauchy-Schwarz inequality once more we have
\begin{multline*}
16\(\int_{B_R}\sqrt\mu\;dy-\sqrt\pi\)^2=\(-\,2\int_{B_R}y\cdot\nabla\sqrt\mu\;dy\)^2\\
=\int_{B_R}\mu\,|y|^2\,dy\int_{B_R}|\nabla\log\mu|^2\,dy\,.
\end{multline*}
This establishes the result in the radial case.

\subsection{Mass transportation for approximating critical Sobolev inequalities}\label{Sec:MassApprox}

Inspired by the limit of Section~\ref{Sec:LimitSobolev2d}, we can indeed obtain Onofri's inequality as a limiting process of critical Sobolev inequalities involving mass transportation. Let us recall the method of \cite{MR2032031}. Let us consider the case where $p<d=2$, 
\[
F=f^\frac{d\,p}{d-p}
\]
and $G$ are two probability measures, $p'=p/(p-1)$ is the H\"older conjugate exponent of $p$ and consider the critical Sobolev inequality
\[
\nr f{\frac{2\,p}{2-p}}^p\leq\mathsf C_{p,d}\,\nr{\nabla f}p^p\quad\forall\,f\in\mathcal D(\R^d)\,.
\]
This inequality generalizes the one in Section~\ref{Sec:LimitSobolev2d} which corresponds to $d=2$ and in particular we have $\mathsf C_{p,2}=\mathsf C_p$. Starting from \eqref{Ineq:CNV}, the proof by mass transportation goes as follows. An integration by parts shows that
\begin{multline*}
\int_{\R^d}G^{1-\frac1d}\,dy\le-\frac{p\,(d-1)}{d\,(d-p)}\int_{\R^d}\nabla(F^{\frac 1p-\frac1d})\cdot F^\frac1{p'}\,\nabla\phi\,dx\\
\le\frac{p\,(d-1)}{d\,(d-p)}\,\nr{\nabla f}p\(\ird{F\,|\nabla\phi|^{p'}}\)^{1/p'}\;,
\end{multline*}
where the last line relies on H\"older's inequality and the fact that $F^{\frac 1p-\frac1d}=f$. The conclusion of the proof arises from the fact that $\ird{F\,|\nabla\phi|^{p'}}=\int_{\R^d}G\,|y|^{p'}\,dy$. It allows to characterize $\mathsf C_{p,d}$ by
\[
\mathsf C_{p,d}=\frac{p\,(d-1)}{d\,(d-p)}\,\inf\frac{\(\int_{\R^d}G\,|y|^{p'}\,dy\)^{1/p'}}{\int_{\R^d}G^{1-\frac1d}\,dy}\;,
\]
where the infimum is taken on all positive probability measures and is achieved by $G=f_\star^\frac{d\,p}{d-p}$. Here $f_\star(x)=(1+|x|^{p'})^{-(d-p)/p}$ is the optimal Aubin-Talenti function.

If we specialize in the case $d=2$ and consider $f=f_\star\,\(1+\tfrac{2-p}{2\,p}\,(u-\bar u)\)$, where $\bar u$ is adjusted so that $\nrm f{\frac{2\,p}{2-p}}=1$, then we recover Onofri's inequality by passing to the limit as $p\to2^-$. Moreover, we may notice that $\nabla(F^{\frac 1p-\frac1d})\cdot F^\frac1{p'}\,\nabla\phi$ formally approaches $\nabla\log F\cdot\sqrt F\,\nabla\phi$, so that the mass transportation method for critical Sobolev inequalities is consistent with the formal computation of Section~\ref{Sec:MassFormal}.

\section{An improved inequality based on Legendre's duality and the logarithmic diffusion or super-fast diffusion equation}\label{Sec:Duality}

In \cite[Theorem 2]{dolbeault:hal-00915998}, it has been shown that
\begin{multline}\label{S-HLS-d=2}
\irdeux{f\,\log\(\frac{f}M\)}-\frac{4\,\pi}M\irdeux{f\,(-\Delta)^{-1}\,f}+M\,(1+\log\pi)\\
\le M\left[\frac1{16\,\pi}\,\nrm{\nabla u}2^2+\irdmu u-\log M\right]
\end{multline}
holds for any function $u\in\mathcal D(\R^2)$ such that $M=\irdmu{e^{\,u}}$ and $f=e^u\mu$. The l.h.s.~in \eqref{S-HLS-d=2} is nonnegative by the logarithmic Hardy-Littlewood-Sobolev type inequality according to \cite[Theorem~1]{MR1143664} (also see \cite[Theorem 2]{MR1230930}).
The inequality \eqref{S-HLS-d=2} is proven by simply expanding the square
\[
0 \le\irdeux{\left|\frac1{8\,\pi}\,\nabla u+\kappa\,\nabla\,(-\Delta)^{-1}(v-\mu)\right|^2}\;,
\]
for some constant $\kappa$ to be appropriately chosen. Alternatively, we may work on the sphere. Let us expand the square
\[
0\le\iS{\left|\frac 12\,\nabla(u-\bar u)+\frac1{\bar v}\,\nabla\,(-\Delta)^{-1}(v-\bar v)\right|^2}\,.
\]
It is then straightforward to see that
\begin{multline*}
\frac14\iS{|\nabla u|^2}+\iS u-\log\(\iS{e^u}\)\\
+\frac1{{\bar v}^2}\iS{(v-\bar v)\,(-\Delta)^{-1}(v-\bar v)}-\frac1{\bar v}\iS{v\,\log\(\frac v{\bar v}\)}\\
\ge\frac2{\bar v}\iS{(u-\bar u)\,(v-\bar v)}\hspace*{4cm}\\
\hspace*{2cm}+\iS u-\log\(\iS{e^u}\)-\frac1{\bar v}\iS{v\,\log\(\frac v{\bar v}\)}\\=:\mathcal R[u,v]\,.
\end{multline*}
Here we assume that
\[
\bar u:=\log\(\iS{e^u}\)\quad\mbox{and}\quad\bar v:=\iS v\,.
\]
With the choice
\[
v=e^u\,,\quad\bar v=e^{\bar u}\,,
\]
the reader is invited to check that $\mathcal R[u,v]=0$. Altogether, we have shown that
\begin{multline*}
\frac14\iS{|\nabla u|^2}+\iS u-\log\(\iS{e^u}\)\\
\ge\iS{f\,\log f}-\iS{(f-1)\,(-\Delta)^{-1}(f-1)}\;,
\end{multline*}
with $f:=e^u/\iS{e^u}$. This inequality is exactly equivalent to \eqref{S-HLS-d=2}. Notice that the r.h.s.~is nonnegative by the logarithmic Hardy-Littlewood-Sobolev inequality, which is the dual inequality of Onofri's. See \cite{MR1143664,MR2915466,dolbeault:hal-00915998} for details.

Keeping track of the square, we arrive at the following identity.
\begin{Prop}\label{Prop:Duality} For any $u\in\H^1(\S^2)$, we have
\begin{multline*}
\frac14\iS{|\nabla u|^2}+\iS u-\log\(\iS{e^u}\)\\
=\iS{f\,\log f}-\iS{(f-1)\,(-\Delta)^{-1}(f-1)}\\
+\iS{\left|\frac 12\,\nabla u+\nabla\,(-\Delta)^{-1}(f-1)\right|^2}\;,
\end{multline*}
with $f:=e^u/\iS{e^u}$.\end{Prop}

It is an open question to get an improved inequality compared to \eqref{S-HLS-d=2} using a flow, as was done in \cite{dolbeault:hal-00915998} for Sobolev and Hardy-Littlewood-Sobolev inequalities. We may for instance consider the logarithmic diffusion equation, which is also called the super-fast diffusion equation, on the two-dimensional sphere $\S^2$
\be{Eqn:super-fast}
\frac{\partial f}{\partial t}=\Lap\log f \;,
\ee
where $\Lap$ denotes the Laplace-Beltrami operator on $\S^2$. In dimension $d=2$ Eq.~\eqref{Eqn:super-fast} plays a role which is the analogue of the Yamabe flow in dimensions $d\ge3$ or, to be precise, to the equation $\frac{\partial f}{\partial t}=\Lap f^\frac{d-2}{d+2}$. See \cite{MR2915466,dolbeault:hal-00915998} for details. The flow defined by \eqref{Eqn:super-fast} does not give straightforward estimates although we may notice that
\[
\mathsf H:=\iS{f\,\log f}-\iS{(f-1)\,(-\Delta)^{-1}(f-1)}
\]
is such that, if $f=e^{u/2}$ is a solution to~\eqref{Eqn:super-fast} such that $\iS f=1$, then
\begin{multline*}
\frac{d\mathsf H}{dt}=-\left[\iS{|\nabla u|^2}+\iS u-\iS{u\,e^{u/2}}\right]\\
\le-\left[\iS{|\nabla u|^2}+\iS u-\log\(\iS{e^u}\)\right]
\end{multline*}
because $\iS{u\,e^{u/2}}\le\log\(\iS{e^u}\)$ according to \cite[Proposition~3.1]{MR2915466}.

\section{An improved inequality based on the entropy--entropy production method and the fast diffusion equation}\label{Sec:Fast-Diffusion}

In $\R^2$ we consider the fast diffusion equation written in self-similar variables
\be{Eqn:FastDiffusion}
\frac{\partial v}{\partial t}+\nabla\cdot\left[v\,\(\nabla v^{m-1}-2\,x\)\right]=0\;,
\ee
where the parameter $m$ is taken in the interval $[1/2,1)$. According to \cite{MR1940370}, the mass $M=\irdeux v$ is independent of $t$. Stationary solutions are the so-called \emph{Barenblatt profiles}
\[
v_\infty(x):=\(D+|x|^2\)^{\frac1{m-1}}\,,
\]
where $D$ is a positive parameter which is uniquely determined by the mass condition $M=\irdeux{v_\infty}$. The \emph{relative entropy} is defined by
\[
\mathcal E[v]:=\frac1{m-1}\irdeux{\left[v^m-v_\infty^m-m\,v_\infty^{m-1}\,(v-v_\infty)\right]}\,.
\]
According to \cite{MR1940370}, it is a Lyapunov functional, since
\[
\frac d{dt}\mathcal E[v]=-\,\mathcal I[v]\;,
\]
where $\mathcal I$ is the \emph{relative Fisher information} defined by
\[
\mathcal I[v]:=\irdeux{v\,|v^{m-1}-v_\infty^{m-1}|^2}\,,
\]
and for $m>\frac12$ the inequality
\be{Ineq:FD}
\mathcal E[v]\le\frac14\,\mathcal I[v]
\ee
is equivalent to a Gagliardo-Nirenberg inequality written with an optimal constant according to \cite{MR1940370}. Note that for $m=1/2$, $v_\infty(x):=\(D+|x|^2\)^{-2}$ and so $v_\infty^m\notin\L^1(\R^2)$ and $|x|^2\,v_\infty\notin\L^1(\R^2)$.

However, we may consider $w=v/v_\infty$ at least for a function $v$ such that~$v-v_\infty$ is compactly supported, take the limit $m\to1/2$ and argue by density to prove that
\[
\mathcal E[w\,v_\infty]=:E[w]=\irdeux{\frac{|\sqrt w-1|^2}{D+|x|^2}}\le\frac14\,I[w] \;,
\]
where
\[
I[w]:=\mathcal I[w\,v_\infty]=\irdeux{v_\infty\,w\,\,\left|\nabla\(v_\infty^{m-1}\,(w^{m-1}-1)\)\right|^2}
\]
can be rewritten as
\begin{align*}
I[w]=\;&\irdeux{\frac w{(D+|x|^2)^2}\,\left|\nabla\(v_\infty^{m-1}\,(w^{m-1}-1)\)\right|^2}\\
=\;&\irdeux{\frac w{(D+|x|^2)^2}\left|\nabla\((D+|x|^2)\,(w^{-1/2}-1)\)\right|^2}\\
=\;&\irdeux{\frac1{(D+|x|^2)^2}\left|\,2\,x\,(1-\sqrt w\,)-\frac12\,(D+|x|^2)\,\nabla\log w\right|^2}\\
=\;&\irdeux{\frac{4\,|x|^2\,(1-\sqrt w\,)^2}{(D+|x|^2)^2}}+\frac14\irdeux{|\nabla\log w|^2}\\
&\hspace*{4cm}-2\irdeux{x\cdot\frac{\nabla\log w+2\,\nabla(1-\sqrt w\,)}{D+|x|^2}}\\
=\;&\irdeux{\frac{4\,|x|^2\,(1-\sqrt w\,)^2}{(D+|x|^2)^2}}+\frac14\irdeux{|\nabla\log w|^2}\\
&\hspace*{4cm}+4\,D\irdeux{\frac{\log w+2\,(1-\sqrt w\,)}{(D+|x|^2)^2}} \;,
\end{align*}
where we performed an integration by parts in the last line. Collecting terms and letting $u=\log w$, we arrive at
\begin{align*}
\frac14\,I[w]-E[w]=\;&-\,D\irdeux{\frac{(1-\sqrt w\,)^2}{(D+|x|^2)^2}}+\frac1{16}\irdeux{|\nabla\log w|^2}\\
&\hspace*{4cm}+D\irdeux{\frac{\log w-2\,(\sqrt w-1)}{(D+|x|^2)^2}}\\
=\;&-D\irdeux{\frac{(1-e^{u/2})^2}{(D+|x|^2)^2}}+\frac1{16}\irdeux{|\nabla u|^2}\\
&\hspace*{4cm}+D\irdeux{\frac{u-2\,(e^{u/2}-1)}{(D+|x|^2)^2}}\\
=\;&\frac1{16}\irdeux{|\nabla u|^2}-D\irdeux{\frac{e^u-1-u}{(D+|x|^2)^2}}\;,
\end{align*}
and thus prove that~\eqref{Ineq:FD} written for $m=1/2$ shows that the r.h.s. in the above identity is nonnegative. As a special case consider $D=1$ and define $d\mu=\mu(x)\,dx$ where $\mu(x)=\frac1\pi\,(1+|x|^2)^{-2}$. Inequality~\eqref{Ineq:FD} can therefore be written as
\[
\frac1{16\,\pi}\irdeux{|\nabla u|^2}\ge\irdmu{e^u}-1-\irdmu u\,.
\]
Since $z-1\ge\log z$ for any $z>0$, this inequality implies the Onofri inequality~\eqref{Onofri:Euclidean}, namely,
\[
\frac1{16\,\pi}\irdeux{|\nabla u|^2}\ge\log\(\irdmu{e^u}\)-\irdmu u\,.
\]
The two inequalities are actually equivalent since the first one is not invariant under a shift by a given constant: if we replace $u$ by $u+c$ with~$c$ such that
\[
  \irdmu{e^u}-1-\irdmu u \ge e^c\irdmu{e^u}-1-\irdmu u-c\,,
\]
and minimize the r.h.s. with respect to~$c$, we get that $c=-\log\(\irdmu{e^u}\)$ and recover the standard form~\eqref{Onofri:Euclidean} of Onofri's inequality.

Various methods are available for proving~\eqref{Ineq:FD}. The Bakry-Emery method, or \emph{carr\'e du champ} method, has been developed in \cite{MR772092,AMTU} in the linear case and later extended to nonlinear diffusions in \cite{MR1940370,MR1777035,MR1853037} using a \emph{relative entropy} which appears first in \cite{MR760591,MR760592}. This \emph{entropy--entropy production} method has the advantage of providing an integral remainder term. Here we adopt a setting that can be found in \cite{MR3103175}.

Let us consider a solution $v$ to \eqref{Eqn:FastDiffusion} and define
\[
z(x,t):=\nabla v^{m-1}-2\,x\;,
\]
so that \eqref{Eqn:FastDiffusion} can be rewritten for any $m\in[\frac12,1)$ as
\[
\frac{\partial v}{\partial t}+\nabla\cdot\(v\,z\)=0\,.
\]
A tedious computation shows that
\[
\frac d{dt}\irdeux{v\,|z|^2}+4\irdeux{v\,|z|^2}=-2\,\frac{1-m}m\,\mathcal R[v,z]\;,
\]
with
\be{Eqn:R}
R[v,z]:=\irdeux{v^m\,\left[\,|\nabla z|^2-(1-m)\,\(\nabla\cdot z\)^2\right]}\;,
\ee
where $|\nabla z|^2=\sum_{i,j=1,2}(\frac{\partial z_i}{\partial x_j})^2$ and $\nabla\cdot z=\sum_{i=1,2}\frac{\partial z_i}{\partial x_i}$. Summarizing, when $m=\frac12$, we have shown that
\[
\frac14\,I[w(t=0,\cdot)]-E[w(t=0,\cdot)]=2\int_0^\infty\mathcal R[v(t,\cdot),z(t,\cdot)]\,dt\,.
\]
\begin{Prop}\label{Prop:Bakry-Emery} If we denote by $v$ the solution to \eqref{Eqn:FastDiffusion} with initial datum
\[
v_{|t=0}=\frac{e^u}{(1+|x|^2)^2}\;,
\]
then we have the identity
\[
\frac1{16\,\pi}\irdeux{|\nabla u|^2}+\irdmu u-\log\(\irdmu{e^u}\)=2\int_0^\infty\mathcal R[v,z]\,dt \;,
\]
with $\mathcal R$ defined by \eqref{Eqn:R} and $z(t,x)=\nabla v^{-1/2}(t,x)-2\,x$.
\end{Prop}
Notice that the kernel of $\mathcal R$ is spanned by all Barenblatt profiles, which are the stationary solutions of \eqref{Eqn:FastDiffusion} (one has to take into account the invariances: multiplication by a constant, translation and dilation). This has to do with the conformal transformation on the sphere: see Theorem~\ref{Thm:Rigidity} and \cite[Section~17.3]{MR3052352} for more details. 

As a straightforward consequence of Propostion~\ref{Prop:Bakry-Emery}, we have the
\begin{Cor}\label{Cor:Bakry-Emery} With the notations of Section~\ref{Sec:Linearization} we have
\[
\mathcal I_1=1\,.
\]
Moreover any minimizing sequence converges to a function in the kernel of~$\mathcal R$.\end{Cor}
The fact that Onofri's inequality is intimately related with the fast diffusion equation~\eqref{Eqn:FastDiffusion} with $m=1/2$ sheds a new light on the role played by this equation for the dual inequality, the logarithmic Hardy-Littlewood-Sobolev inequality, which has been studied in~\cite{MR2745814} and applied to the critical parabolic-elliptic Keller-Segel model in~\cite{MR2876403,MR3024094}.

\section{Rigidity (or \emph{carr\'e du champ}) methods and adapted nonlinear diffusion equations}\label{Sec:Rigidity}

By \emph{rigidity method} we refer to a method which has been popularized in \cite{MR615628} and optimized later in \cite{MR1134481}. We will first consider the symmetric case in which computations can be done along the lines of \cite{DEKL} and are easy. Then we will introduce flows as in \cite{DEKL} (for Sobolev's inequality and interpolation inequalities in the subcritical range), still in the symmetric case. The main advantage is that the flow produces an integral remainder term which is, as far as we know, a new result in the case of Onofri's inequality.

The integrations by parts of the rigidity method can be encoded in the $\Gamma_2$ or \emph{carr\'e du champ} methods, thus providing the same results. In the case of Onofri's inequality, this has been observed by \'E.~Fontenas in \cite[Theorem ~2]{MR1651226} (actually, without symmetry).

A striking observation is indeed that no symmetry is required: the rigidity computations and the flow can be used in the general case, as was done in \cite{1302}, and produce an integral remainder term, which is our last new result.

\subsection{Rigidity method in the symmetric case}\label{Sec:RigiditySym}

As shown for instance in \cite{MR960228} the functional
\[
\mathcal G_\lambda[v]:=\frac14\iS{|\nabla v|^2}+\lambda\,\left[\iS{v}-\log\(\iS{e^v}\)\right]
\]
is nonnegative for all $\lambda>0$ and it can be minimized in $\H^1(\S^2)$ and, up to the addition of a constant, any minimizer satisfies the Euler-Lagrange equation
\be{Eqn:EulerLagrange}
-\frac12\,\Delta v+\lambda=\lambda\,e^v\quad\mbox{on}\quad S^2\,.
\ee
According to Proposition~\ref{Prop:Symmetry}, minimizing $\mathcal G_\lambda$ amounts to minimizing
\[
\mathsf G_\lambda[f]:=\frac18\int_{-1}^1|f'|^2\,\nu\,dz+\frac\lambda2\int_{-1}^1f\,dz\ge\lambda\,\log\(\frac12\int_{-1}^1e^f\,dz\)\,,
\]
and \eqref{Ultraspherical} can be reduced to the fact that the minimum of $\mathsf G_1$ is achieved by constant functions. For the same reasons as above, $\mathsf G_\lambda$ has a minimum which solves the Euler-Lagrange equation
\[
- \frac12\,\LL f+\lambda=2\,\lambda\,\frac{e^f}{\int_{-1}^1e^f\,dz}\;,
\]
where $\LL f:=\nu\,f''+\nu'\,f'$ and $\nu(z)=1-z^2$. Up to the addition of a constant, we may choose $f$ such that $\int_{-1}^1e^f\,dz=2$ and hence solves
\be{equ}
-\frac12\,\LL f+\lambda=\lambda\,e^f\,.
\ee
\begin{theorem}\label{Thm:Rigidity} For any $\lambda\in(0,1)$, \eqref{equ} has a unique smooth solution $f$, which is the constant function
\[
f=0\,.
\]
As a consequence, if $f$ is a critical point of the functional $\mathsf G_\lambda$, then $f$ is a constant function for any $\lambda\in(0,1)$, while for $\lambda=1$, $f$ has to satisfy the differential equation $f''=\frac12\,|f'|^2$ and is either a constant, or such that
\be{exact-conformal}
f(z)=C_1-2\,\log(C_2-z)\,,
\ee
for some constants $C_1\in\R$ and $C_2>1$.\end{theorem}
Let us define
\be{restee}
\mathsf R_\lambda[f]:=\frac18\izp{\nu^2\,\left|f''-\tfrac12\,|f'|^2\right|^2\,e^{-f/2}}+\frac{1-\lambda}4\izp{\nu\,|f'|^2\,e^{-f/2}}\,.
\ee
The proof is a straightforward consequence of the following lemma.
\begin{Lem} If $f$ solves \eqref{equ}, then
\[
\mathsf R_\lambda[f]=0\,.
\]\end{Lem}
\begin{proof} The ultraspherical operator does not commute with the derivation with respect to $z$:
\[(\LL f)'= \LL f'-2\,z\,f''-2\,f' \;,
\]
where $f'=\frac{df}{dz}$. After multiplying \eqref{equ} by $\LL\big(e^{-f/2}\big)$ and integrating by parts, we get
\begin{multline*}
0=\izp{\(-\tfrac12\,\LL f+\lambda-e^f\)\,\LL\big(e^{-f/2}\big)}\\
=\frac14\izp{\nu^2\,|f''|^2\,e^{-f/2}}-\frac18\izp{\nu^2\,|f'|^2\,f''\,e^{-f/2}}\\
+\frac12\izp{\nu\,|f'|^2\,e^{-f/2}}-\frac12\izp{\nu\,|f'|^2\,e^{f/2}}\,.
\end{multline*}
Similarly, after multiplying \eqref{equ} by $\frac\nu2\,|f'|^2\,e^{-f/2}$ and integrating by parts, we get
\begin{multline*}
0=\izp{\(-\tfrac12\,\LL f+\lambda-e^f\)\,\(\tfrac\nu{2}\,|f'|^2\,e^{-f/2}\)}\\
=\frac18\izp{\nu^2\,|f'|^2\,f''\,e^{-f/2}}-\frac1{16}\izp{\nu^2\,|f'|^4\,e^{-f/2}}\\
+\frac\lambda2\izp{\nu\,|f'|^2\,e^{-f/2}}-\frac12\izp{\nu\,|f'|^2\,e^{f/2}}\,.
\end{multline*}
Subtracting the second identity from the first one establishes the first part of the theorem. If $\lambda\in(0,1)$, then $f$ has to be a constant. If $\lambda=1$, there are other solutions, because of the conformal transformations: see for instance \cite[Section~17.3]{MR3052352} for more details. In our case, all solutions of the differential equation $f''=\frac12\,|f'|^2$ that are not constant are given by \eqref{exact-conformal}.
\end{proof}

\subsection{A nonlinear flow method in the symmetric case}\label{Sec:Flow}

Consider the nonlinear evolution equation
\be{Eqn:EvolSym}
\frac{\partial g}{\partial t}=\LL\,(e^{-g/2})-\tfrac\nu2\,|g'|^2\,e^{-g/2}\,.
\ee
\begin{Prop}\label{Prop:EvolSym} Assume that $g$ is a solution to \eqref{Eqn:EvolSym} with initial datum $f\in\L^1(-1,1;dz)$ such that $\int_{-1}^1|f'|^2\,\nu\,dz$ is finite and $\izp{e^f}=1$. Then for any $\lambda\in(0,1]$ we have
\[
\mathsf G_\lambda[f]\ge\int_0^\infty\mathsf R_\lambda[g(t,\cdot)]\,dt\,,
\]
where $R_\lambda$ is defined in \eqref{restee}.
\end{Prop}
\begin{proof} A standard regularization method allows to reduce the evolution problem to the case of smooth bounded functions, at least in a finite time interval. Then a simple computation shows that
\[
\frac d{dt}\mathsf G_\lambda[g(t,\cdot)]=-\frac12\izp{\(-\tfrac12\,\LL g+\lambda-\lambda\,e^g\)\,\frac{\partial g}{\partial t}}=-\mathsf R_\lambda[g(t,\cdot)]\,.
\]
We may then argue by continuation. Because $\mathsf G_\lambda[g(t,\cdot)]$ is bounded from below, $\mathsf R_\lambda[g(t,\cdot)]$ is integrable with respect to $t\in[0,\infty)$. Hence, as $t\to\infty$, $g$ converges to a constant if $\lambda<1$, or the conformal transformation of a constant if $\lambda=1$ and therefore $\lim_{t\to\infty}\mathsf G_\lambda[g(t,\cdot)]=0$. The result holds with equality after integrating on $[0,\infty)\ni t$. For a general initial datum without smoothness assumption we conclude by density and get an inequality instead of an equality by lower semi-continuity. \end{proof}

For a general function $v\in\H^1(\S^2)$, if we denote by $v_*$ the symmetrized function which depends only on $\theta$ (see \cite[Section~17.1]{MR3052352} for more details) and denote by $f$ the function such that $f(\cos\theta)=v_*(\theta)$, then it follows from Propositions~\ref{Prop:Symmetry} and~\ref{Prop:EvolSym} that 
\[
\mathcal G_\lambda[v]\ge\int_0^\infty\mathsf R_\lambda[g(t,\cdot)]\,dt \;,
\]
where $g$ is the solution to \eqref{Eqn:EvolSym} with initial datum $f$. However, we do not need any symmetrization step, as we shall see in the next section.

\subsection{A nonlinear flow method in the general case}\label{Sec:General}

On $\S^2$ let us consider the nonlinear evolution equation
\be{Eqn:Evol}
\frac{\partial f}{\partial t}=\Delta_{\S^2}\,(e^{-f/2})-\tfrac12\,|\nabla f|^2\,e^{-f/2}\;,
\ee
where $\Delta_{\S^2}$ denotes the Laplace-Beltrami operator. Let us define
\[
\mathcal R_\lambda[f]:=\frac12\iS{\|\mathrm L_{\S^2}f-\frac12\,\mathrm M_{\S^2}f\|^2\,e^{-f/2}}+\frac12\,(1-\lambda)\iS{|\nabla f|^2\,e^{-f/2}}\;,
\]
where
\[
\mathrm L_{\S^2}f:=\mathrm{Hess}_{\S^2}\,f-\frac12\,\Delta_{\S^2}f\,\mathrm{Id}\quad\mbox{and}\quad\mathrm M_{\S^2}f:=\nabla f\otimes\nabla f-\frac12\,|\nabla f|^2\,\mathrm{Id}\,.
\]
This definition of $\mathcal R_\lambda$ generalizes the definition of $\mathsf R_\lambda$ given in Section~\ref{Sec:RigiditySym} in the symmetric case. We refer to \cite{1302} for more detailed considerations, and to \cite{DEKL} for considerations and improvements of the method that are specific to the sphere $\S^2$.
\begin{theorem}\label{Thm:Evol} Assume that $f$ is a solution to \eqref{Eqn:Evol} with initial datum $v-\log\(\iS{e^v}\)$, where $v\in\L^1(\S^2)$ is such that $\nabla v\in\L^2(\S^2)$. Then for any $\lambda\in(0,1]$ we have
\[
\mathcal G_\lambda[v]\ge\int_0^\infty\mathcal R_\lambda[f(t,\cdot)]\,dt\,.
\]
\end{theorem}
\begin{proof} With no restriction, we may assume that $\iS{e^v}=1$ and it is then straightforward to see that $\iS{e^{f(t,\cdot)}}=1$ for any $t>0$. Next we compute
\[
\frac d{dt}\mathcal G_\lambda[f]=\iS{\(-\tfrac12\,\Delta_{\S^2}f+\lambda\)\(\Delta_{\S^2}\,(e^{-f/2})-\tfrac12\,|\nabla f|^2\,e^{-f/2}\)}=-\,\mathcal R_\lambda[f]
\]
in the same spirit as in \cite{1302}.\end{proof}

As a concluding remark, let us notice that the \emph{carr\'e du champ} method is not limited to the case of $\S^2$ but also applies to two-dimensional Riemannian manifolds: see for instance \cite{MR1435336}. The use of the flow defined by \eqref{Eqn:Evol} gives an additional integral remainder term, in the spirit of what has been done in \cite{1302}. This is however out of the scope of the present paper.

\medskip\begin{spacing}{0.75}\noindent{\small{\bf Acknowlegments.} J.D.~and G.J.~have been supported by the projects \emph{STAB} and \emph{Kibord} of the French National Research Agency (ANR). J.D.~and M.J.E..~have been supported by the project \emph{NoNAP} of the French National Research Agency (ANR). J.D.~and G.J.~thank M.~del Pino and M.~Kowalczyk for stimulating discussions during exchanges that were supported by the ECOS project n$^\circ$ C11E07 on \emph{Functional inequalities, asymptotics and dynamics of fronts}.
The authors also thank Lucilla Corrias and Shirin Boroushaki for their helpful remarks and comments.
}\\
{\sl\scriptsize\copyright~2015 by the authors. This paper may be reproduced, in its entirety, for non-commercial purposes.}\end{spacing}

\begin{thebibliography}{69}
\newcommand{\enquote}[1]{#1}
\providecommand{\natexlab}[1]{#1}
\providecommand{\url}[1]{\texttt{#1}}
\providecommand{\urlprefix}{URL }
\providecommand{\eprint}{eprint }
\expandafter\ifx\csname urlstyle\endcsname\relax
 \providecommand{\doi}[1]{doi:\discretionary{}{}{}#1}\else
 \providecommand{\doi}{doi:\discretionary{}{}{}\begingroup
 \urlstyle{rm}\Url}\fi

\bibitem[{Adachi and Tanaka(2000)}]{MR1646323}
Adachi, S. and Tanaka, K. (2000). \enquote{Trudinger type inequalities in
 {$\mathbb R\sp N$} and their best exponents,} \emph{Proc. Amer. Math. Soc.}
 \textbf{128}, 7, pp. 2051--2057, \doi{10.1090/S0002-9939-99-05180-1}.

\bibitem[{Arnold \emph{et~al.}(2001)Arnold, Markowich, Toscani and
 Unterreiter}]{AMTU}
Arnold, A., Markowich, P., Toscani, G., and Unterreiter, A. (2001).
 \enquote{On convex {S}obolev inequalities and the rate of convergence to
 equilibrium for {F}okker-{P}lanck type equations,} \emph{Comm. Partial
 Differential Equations} \textbf{26}, 1-2, pp. 43--100.

\bibitem[{Aubin(1976)}]{MR0448404}
Aubin, T. (1976). \enquote{Probl\`emes isop\'erim\'etriques et espaces de
 {S}obolev,} \emph{J. Differential Geometry} \textbf{11}, 4, pp. 573--598.

\bibitem[{Aubin(1979)}]{MR534672}
Aubin, T. (1979). \enquote{Meilleures constantes dans le th\'eor\`eme
 d'inclusion de {S}obolev et un th\'eor\`eme de {F}redholm non lin\'eaire pour
 la transformation conforme de la courbure scalaire,} \emph{J. Funct. Anal.}
 \textbf{32}, 2, pp. 148--174, \doi{10.1016/0022-1236(79)90052-1}.

\bibitem[{Baernstein(1994)}]{MR1297773}
Baernstein, A., II (1994). \enquote{A unified approach to symmetrization,} in
 \emph{Partial differential equations of elliptic type ({C}ortona, 1992)},
 Sympos. Math., XXXV (Cambridge Univ. Press, Cambridge), pp. 47--91.

\bibitem[{Baernstein and Taylor(1976)}]{MR0402083}
Baernstein, A., II and Taylor, B.~A. (1976). \enquote{Spherical rearrangements,
 subharmonic functions, and {$\sp*$}-functions in {$n$}-space,} \emph{Duke
 Math. J.} \textbf{43}, 2, pp. 245--268.

\bibitem[{Bakry and {\'E}mery(1984)}]{MR772092}
Bakry, D. and {\'E}mery, M. (1984). \enquote{Hypercontractivit\'e de
 semi-groupes de diffusion,} \emph{C. R. Acad. Sci. Paris S\'er. I Math.}
 \textbf{299}, 15, pp. 775--778.

\bibitem[{Beckner(1993)}]{MR1230930}
Beckner, W. (1993). \enquote{Sharp {S}obolev inequalities on the sphere and the
 {M}oser-{T}rudinger inequality,} \emph{Ann. of Math. (2)} \textbf{138}, 1,
 pp. 213--242, \doi{10.2307/2946638}.

\bibitem[{Bentaleb(1993)}]{MR1231419}
Bentaleb, A. (1993). \enquote{In\'egalit\'e de {S}obolev pour l'op\'erateur
 ultrasph\'erique,} \emph{C. R. Acad. Sci. Paris S\'er. I Math.} \textbf{317},
 2, pp. 187--190.

\bibitem[{Bidaut-V{\'e}ron and V{\'e}ron(1991)}]{MR1134481}
Bidaut-V{\'e}ron, M.-F. and V{\'e}ron, L. (1991). \enquote{Nonlinear elliptic
 equations on compact {R}iemannian manifolds and asymptotics of {E}mden
 equations,} \emph{Invent. Math.} \textbf{106}, 3, pp. 489--539,
 \doi{10.1007/BF01243922}.

\bibitem[{Blanchet \emph{et~al.}(2012)Blanchet, Carlen and
 Carrillo}]{MR2876403}
Blanchet, A., Carlen, E.~A., and Carrillo, J.~A. (2012). \enquote{Functional
 inequalities, thick tails and asymptotics for the critical mass
 {P}atlak-{K}eller-{S}egel model,} \emph{J. Funct. Anal.} \textbf{262}, 5, pp.
 2142--2230, \doi{10.1016/j.jfa.2011.12.012}.

\bibitem[{Bliss(1930)}]{bliss1930integral}
Bliss, G. (1930). \enquote{{An integral inequality},} \emph{Journal of the
 London Mathematical Society} \textbf{1}, 1, p.~40.

\bibitem[{Branson \emph{et~al.}(2007)Branson, Fontana and
 Morpurgo}]{branson712moser}
Branson, T., Fontana, L., and Morpurgo, C. (2007). \enquote{{Moser-Trudinger
 and Beckner-Onofri's inequalities on the CR sphere},} \emph{arXiv}
 \textbf{712}.

\bibitem[{Brothers and Ziemer(1988)}]{MR929981}
Brothers, J.~E. and Ziemer, W.~P. (1988). \enquote{Minimal rearrangements of
 {S}obolev functions,} \emph{J. Reine Angew. Math.} \textbf{384}, pp.
 153--179.

\bibitem[{Caffarelli \emph{et~al.}(1984)Caffarelli, Kohn and
 Nirenberg}]{MR768824}
Caffarelli, L., Kohn, R., and Nirenberg, L. (1984). \enquote{First order
 interpolation inequalities with weights,} \emph{Compositio Math.}
 \textbf{53}, 3, pp. 259--275.

\bibitem[{Calvez and Corrias(2008)}]{MR2433703}
Calvez, V. and Corrias, L. (2008). \enquote{The parabolic-parabolic
 {K}eller-{S}egel model in {$\mathbb R^2$},} \emph{Commun. Math. Sci.}
 \textbf{6}, 2, pp. 417--447.

\bibitem[{Carlen \emph{et~al.}(2010)Carlen, Carrillo and Loss}]{MR2745814}
Carlen, E.~A., Carrillo, J.~A., and Loss, M. (2010).
 \enquote{Hardy-{L}ittlewood-{S}obolev inequalities via fast diffusion flows,}
 \emph{Proc. Natl. Acad. Sci. USA} \textbf{107}, 46, pp. 19696--19701,
 \doi{10.1073/pnas.1008323107}.

\bibitem[{Carlen and Figalli(2013)}]{MR3024094}
Carlen, E.~A. and Figalli, A. (2013). \enquote{Stability for a {GNS} inequality
 and the log-{HLS} inequality, with application to the critical mass
 {K}eller-{S}egel equation,} \emph{Duke Math. J.} \textbf{162}, 3, pp.
 579--625, \doi{10.1215/00127094-2019931}.

\bibitem[{Carlen and Loss(1992)}]{MR1143664}
Carlen, E.~A. and Loss, M. (1992). \enquote{Competing symmetries, the
 logarithmic {HLS} inequality and {O}nofri's inequality on {$\mathbb S\sp
 n$},} \emph{Geom. Funct. Anal.} \textbf{2}, 1, pp. 90--104,
 \doi{10.1007/BF01895706}.

\bibitem[{Carleson and Chang(1986)}]{MR878016}
Carleson, L. and Chang, S.-Y.~A. (1986). \enquote{On the existence of an
 extremal function for an inequality of {J}.\ {M}oser,} \emph{Bull. Sci. Math.
 (2)} \textbf{110}, 2, pp. 113--127.

\bibitem[{Carrillo \emph{et~al.}(2001)Carrillo, J{\"u}ngel, Markowich, Toscani
 and Unterreiter}]{MR1853037}
Carrillo, J.~A., J{\"u}ngel, A., Markowich, P.~A., Toscani, G., and
 Unterreiter, A. (2001). \enquote{Entropy dissipation methods for degenerate
 parabolic problems and generalized {S}obolev inequalities,} \emph{Monatsh.
 Math.} \textbf{133}, 1, pp. 1--82, \doi{10.1007/s006050170032}.

\bibitem[{Carrillo and Toscani(2000)}]{MR1777035}
Carrillo, J.~A. and Toscani, G. (2000). \enquote{Asymptotic {$\mathrm
 L^1$}-decay of solutions of the porous medium equation to self-similarity,}
 \emph{Indiana Univ. Math. J.} \textbf{49}, 1, pp. 113--142,
 \doi{10.1512/iumj.2000.49.1756}.

\bibitem[{Chang(1987)}]{MR934274}
Chang, S.-Y.~A. (1987). \enquote{Extremal functions in a sharp form of
 {S}obolev inequality,} in \emph{Proceedings of the {I}nternational {C}ongress
 of {M}athematicians, {V}ol. 1, 2 ({B}erkeley, {C}alif., 1986)} (Amer. Math.
 Soc., Providence, RI), pp. 715--723.

\bibitem[{Chang and Yang(1987)}]{MR908146}
Chang, S.-Y.~A. and Yang, P.~C. (1987). \enquote{Prescribing {G}aussian
 curvature on {$S^2$},} \emph{Acta Math.} \textbf{159}, 3-4, pp. 215--259,
 \doi{10.1007/BF02392560}.

\bibitem[{Chang and Yang(1988)}]{MR925123}
Chang, S.-Y.~A. and Yang, P.~C. (1988). \enquote{Conformal deformation of
 metrics on {$S^2$},} \emph{J. Differential Geom.} \textbf{27}, 2, pp.
 259--296.

\bibitem[{Chang and Yang(2003)}]{MR1989228}
Chang, S.-Y.~A. and Yang, P.~C. (2003). \enquote{The inequality of {M}oser and
 {T}rudinger and applications to conformal geometry,} \emph{Comm. Pure Appl.
 Math.} \textbf{56}, 8, pp. 1135--1150, \doi{10.1002/cpa.3029}, dedicated to
 the memory of J{\"u}rgen K. Moser.

\bibitem[{Cordero-Erausquin \emph{et~al.}(2004)Cordero-Erausquin, Nazaret and
 Villani}]{MR2032031}
Cordero-Erausquin, D., Nazaret, B., and Villani, C. (2004). \enquote{A
 mass-transportation approach to sharp {S}obolev and {G}agliardo-{N}irenberg
 inequalities,} \emph{Adv. Math.} \textbf{182}, 2, pp. 307--332,
 \doi{10.1016/S0001-8708(03)00080-X}.

\bibitem[{del Pino and Dolbeault(2002)}]{MR1940370}
del Pino, M. and Dolbeault, J. (2002). \enquote{Best constants for
 {G}agliardo-{N}irenberg inequalities and applications to nonlinear
 diffusions,} \emph{J. Math. Pures Appl. (9)} \textbf{81}, 9, pp. 847--875,
 \doi{10.1016/S0021-7824(02)01266-7}.

\bibitem[{del Pino and Dolbeault(2013)}]{MR3089736}
del Pino, M. and Dolbeault, J. (2013). \enquote{The {E}uclidean {O}nofri
 inequality in higher dimensions,} \emph{Int. Math. Res. Not. IMRN}
 \textbf{15}, pp. 3600--3611.

\bibitem[{Dolbeault(2011)}]{MR2915466}
Dolbeault, J. (2011). \enquote{Sobolev and {H}ardy-{L}ittlewood-{S}obolev
 inequalities: duality and fast diffusion,} \emph{Math. Res. Lett.}
 \textbf{18}, 6, pp. 1037--1050.

\bibitem[{Dolbeault \emph{et~al.}(2013{\natexlab{a}})Dolbeault, Esteban,
 Kowalczyk and Loss}]{DEKL2012}
Dolbeault, J., Esteban, M.~J., Kowalczyk, M., and Loss, M.
 (2013{\natexlab{a}}). \enquote{Sharp interpolation inequalities on the
 sphere: New methods and consequences,} \emph{Chinese Annals of Mathematics,
 Series B} \textbf{34}, 1, pp. 99--112.

\bibitem[{Dolbeault \emph{et~al.}(2014)Dolbeault, Esteban, Kowalczyk and
 Loss}]{DEKL}
Dolbeault, J., Esteban, M.~J., Kowalczyk, M., and Loss, M. (2014).
 \enquote{Improved interpolation inequalities on the sphere,} \emph{Discrete
 and Continuous Dynamical Systems Series S (DCDS-S)} \textbf{7}, 4, pp.
 695--724.

\bibitem[{Dolbeault \emph{et~al.}(2013{\natexlab{b}})Dolbeault, Esteban and
 Laptev}]{1301}
Dolbeault, J., Esteban, M.~J., and Laptev, A. (2013{\natexlab{b}}).
 \enquote{{Spectral estimates on the sphere},} Tech. rep., Preprint Ceremade.

\bibitem[{Dolbeault \emph{et~al.}(2013{\natexlab{c}})Dolbeault, Esteban, Laptev
 and Loss}]{1307}
Dolbeault, J., Esteban, M.~J., Laptev, A., and Loss, M. (2013{\natexlab{c}}).
 \enquote{{One-dimensional Gagliardo-Nirenberg-Sobolev inequalities: Remarks
 on duality and flows},} Tech. rep., Preprint Ceremade.

\bibitem[{Dolbeault \emph{et~al.}(2013{\natexlab{d}})Dolbeault, Esteban, Laptev
 and Loss}]{Dolbeault2013437}
Dolbeault, J., Esteban, M.~J., Laptev, A., and Loss, M. (2013{\natexlab{d}}).
 \enquote{Spectral properties of {S}chr{\"o}dinger operators on compact
 manifolds: Rigidity, flows, interpolation and spectral estimates,}
 \emph{Comptes Rendus Mathematique} \textbf{351}, 11--12, pp. 437 -- 440.

\bibitem[{Dolbeault \emph{et~al.}(2013{\natexlab{e}})Dolbeault, Esteban and
 Loss}]{1302}
Dolbeault, J., Esteban, M.~J., and Loss, M. (2013{\natexlab{e}}).
 \enquote{{Nonlinear flows and rigidity results on compact manifolds},} Tech.
 rep., Preprint Ceremade.

\bibitem[{Dolbeault \emph{et~al.}(2008)Dolbeault, Esteban and
 Tarantello}]{MR2437030}
Dolbeault, J., Esteban, M.~J., and Tarantello, G. (2008). \enquote{The role of
 {O}nofri type inequalities in the symmetry properties of extremals for
 {C}affarelli-{K}ohn-{N}irenberg inequalities, in two space dimensions,}
 \emph{Ann. Sc. Norm. Super. Pisa Cl. Sci. (5)} \textbf{7}, 2, pp. 313--341.

\bibitem[{Dolbeault \emph{et~al.}(2009)Dolbeault, Esteban and
 Tarantello}]{MR2509375}
Dolbeault, J., Esteban, M.~J., and Tarantello, G. (2009).
 \enquote{Multiplicity results for the assigned {G}auss curvature problem in
 {$\mathbb R^2$},} \emph{Nonlinear Anal.} \textbf{70}, 8, pp. 2870--2881,
 \doi{10.1016/j.na.2008.12.040}.

\bibitem[{Dolbeault and Jankowiak(2014)}]{dolbeault:hal-00915998}
Dolbeault, J. and Jankowiak, G. (2014). \enquote{{Sobolev and
 Hardy-Littlewood-Sobolev inequalities},} Preprint.

\bibitem[{Dolbeault and Toscani(2013)}]{MR3103175}
Dolbeault, J. and Toscani, G. (2013). \enquote{Improved interpolation
 inequalities, relative entropy and fast diffusion equations,} \emph{Ann.
 Inst. H. Poincar\'e Anal. Non Lin\'eaire} \textbf{30}, 5, pp. 917--934,
 \doi{10.1016/j.anihpc.2012.12.004}.

\bibitem[{Flucher(1992)}]{MR1171306}
Flucher, M. (1992). \enquote{Extremal functions for the {T}rudinger-{M}oser
 inequality in {$2$} dimensions,} \emph{Comment. Math. Helv.} \textbf{67}, 3,
 pp. 471--497, \doi{10.1007/BF02566514}.

\bibitem[{Fontenas(1997)}]{MR1435336}
Fontenas, {\'E}. (1997). \enquote{Sur les constantes de {S}obolev des
 vari\'et\'es riemanniennes compactes et les fonctions extr\'emales des
 sph\`eres,} \emph{Bull. Sci. Math.} \textbf{121}, 2, pp. 71--96.

\bibitem[{Fontenas(1998)}]{MR1651226}
Fontenas, {\'E}. (1998). \enquote{Sur les minorations des constantes de
 {S}obolev et de {S}obolev logarithmiques pour les op\'erateurs de {J}acobi et
 de {L}aguerre,} in \emph{S\'eminaire de {P}robabilit\'es, {XXXII}},
 \emph{Lecture Notes in Math.}, Vol. 1686 (Springer, Berlin), pp. 14--29,
 \doi{10.1007/BFb0101747}.

\bibitem[{Gajewski and Zacharias(1998)}]{MR1654677}
Gajewski, H. and Zacharias, K. (1998). \enquote{Global behaviour of a
 reaction-diffusion system modelling chemotaxis,} \emph{Math. Nachr.}
 \textbf{195}, pp. 77--114, \doi{10.1002/mana.19981950106}.

\bibitem[{Ghigi(2005)}]{MR2154301}
Ghigi, A. (2005). \enquote{On the {M}oser-{O}nofri and {P}r\'ekopa-{L}eindler
 inequalities,} \emph{Collect. Math.} \textbf{56}, 2, pp. 143--156.

\bibitem[{Ghoussoub and Lin(2010)}]{MR2670931}
Ghoussoub, N. and Lin, C.-S. (2010). \enquote{On the best constant in the
 {M}oser-{O}nofri-{A}ubin inequality,} \emph{Comm. Math. Phys.} \textbf{298},
 3, pp. 869--878, \doi{10.1007/s00220-010-1079-7}.

\bibitem[{Ghoussoub and Moradifam(2013)}]{MR3052352}
Ghoussoub, N. and Moradifam, A. (2013). \emph{Functional inequalities: new
 perspectives and new applications}, \emph{Mathematical Surveys and
 Monographs}, Vol. 187 (American Mathematical Society, Providence, RI), ISBN
 978-0-8218-9152-0.

\bibitem[{Gidas and Spruck(1981)}]{MR615628}
Gidas, B. and Spruck, J. (1981). \enquote{Global and local behavior of positive
 solutions of nonlinear elliptic equations,} \emph{Comm. Pure Appl. Math.}
 \textbf{34}, 4, pp. 525--598, \doi{10.1002/cpa.3160340406}.

\bibitem[{Hersch(1970)}]{MR0292357}
Hersch, J. (1970). \enquote{Quatre propri\'et\'es isop\'erim\'etriques de
 membranes sph\'eriques homog\`enes,} \emph{C. R. Acad. Sci. Paris S\'er. A-B}
 \textbf{270}, pp. A1645--A1648.

\bibitem[{Hong(1986)}]{MR845999}
Hong, C.~W. (1986). \enquote{A best constant and the {G}aussian curvature,}
 \emph{Proc. Amer. Math. Soc.} \textbf{97}, 4, pp. 737--747,
 \doi{10.2307/2045939}.

\bibitem[{Kawohl(1985)}]{MR810619}
Kawohl, B. (1985). \emph{Rearrangements and convexity of level sets in {PDE}},
 \emph{Lecture Notes in Mathematics}, Vol. 1150 (Springer-Verlag, Berlin),
 ISBN 3-540-15693-3.

\bibitem[{Lam and Lu(2013)}]{MR3053467}
Lam, N. and Lu, G. (2013). \enquote{A new approach to sharp {M}oser-{T}rudinger
 and {A}dams type inequalities: a rearrangement-free argument,} \emph{J.
 Differential Equations} \textbf{255}, 3, pp. 298--325,
 \doi{10.1016/j.jde.2013.04.005}.

\bibitem[{Lieb(1983)}]{MR717827}
Lieb, E.~H. (1983). \enquote{Sharp constants in the
 {H}ardy-{L}ittlewood-{S}obolev and related inequalities,} \emph{Ann. of Math.
 (2)} \textbf{118}, 2, pp. 349--374, \doi{10.2307/2007032}.

\bibitem[{Lieb and Loss(2001)}]{MR1817225}
Lieb, E.~H. and Loss, M. (2001). \emph{Analysis}, \emph{Graduate Studies in
 Mathematics}, Vol.~14, 2nd edn. (American Mathematical Society, Providence,
 RI), ISBN 0-8218-2783-9.

\bibitem[{McCann(1995)}]{MR1369395}
McCann, R.~J. (1995). \enquote{Existence and uniqueness of monotone
 measure-preserving maps,} \emph{Duke Math. J.} \textbf{80}, 2, pp. 309--323,
 \doi{10.1215/S0012-7094-95-08013-2}.

\bibitem[{McCann(1997)}]{MR1451422}
McCann, R.~J. (1997). \enquote{A convexity principle for interacting gases,}
 \emph{Adv. Math.} \textbf{128}, 1, pp. 153--179,
 \doi{10.1006/aima.1997.1634}.

\bibitem[{McLeod and Peletier(1989)}]{MR981664}
McLeod, J.~B. and Peletier, L.~A. (1989). \enquote{Observations on {M}oser's
 inequality,} \emph{Arch. Rational Mech. Anal.} \textbf{106}, 3, pp. 261--285,
 \doi{10.1007/BF00281216}.

\bibitem[{Moser(1970/71)}]{MR0301504}
Moser, J. (1970/71). \enquote{A sharp form of an inequality by {N}.
 {T}rudinger,} \emph{Indiana Univ. Math. J.} \textbf{20}, pp. 1077--1092.

\bibitem[{Newman(1984)}]{MR760591}
Newman, W.~I. (1984). \enquote{A {L}yapunov functional for the evolution of
 solutions to the porous medium equation to self-similarity. {I},} \emph{J.
 Math. Phys.} \textbf{25}, 10, pp. 3120--3123.

\bibitem[{Okikiolu(2008)}]{MR2377499}
Okikiolu, K. (2008). \enquote{Extremals for logarithmic
 {H}ardy-{L}ittlewood-{S}obolev inequalities on compact manifolds,}
 \emph{Geom. Funct. Anal.} \textbf{17}, 5, pp. 1655--1684,
 \doi{10.1007/s00039-007-0636-5}.

\bibitem[{Onofri(1982)}]{MR677001}
Onofri, E. (1982). \enquote{On the positivity of the effective action in a
 theory of random surfaces,} \emph{Comm. Math. Phys.} \textbf{86}, 3, pp.
 321--326.

\bibitem[{Osgood \emph{et~al.}(1988)Osgood, Phillips and Sarnak}]{MR960228}
Osgood, B., Phillips, R., and Sarnak, P. (1988). \enquote{Extremals of
 determinants of {L}aplacians,} \emph{J. Funct. Anal.} \textbf{80}, 1, pp.
 148--211, \doi{10.1016/0022-1236(88)90070-5}.

\bibitem[{Ralston(1984)}]{MR760592}
Ralston, J. (1984). \enquote{A {L}yapunov functional for the evolution of
 solutions to the porous medium equation to self-similarity. {II},} \emph{J.
 Math. Phys.} \textbf{25}, 10, pp. 3124--3127.

\bibitem[{Rosen(1971)}]{MR0289739}
Rosen, G. (1971). \enquote{Minimum value for {$c$} in the {S}obolev inequality
 {$\Vert \phi\sp{3}\Vert \leq c\,\Vert \nabla \phi\Vert \sp{3}$},} \emph{SIAM
 J. Appl. Math.} \textbf{21}, pp. 30--32.

\bibitem[{Rubinstein(2008{\natexlab{a}})}]{MR2473271}
Rubinstein, Y.~A. (2008{\natexlab{a}}). \enquote{On energy functionals,
 {K}\"ahler-{E}instein metrics, and the {M}oser-{T}rudinger-{O}nofri
 neighborhood,} \emph{J. Funct. Anal.} \textbf{255}, 9, pp. 2641--2660,
 \doi{10.1016/j.jfa.2007.10.009}.

\bibitem[{Rubinstein(2008{\natexlab{b}})}]{MR2419932}
Rubinstein, Y.~A. (2008{\natexlab{b}}). \enquote{Some discretizations of
 geometric evolution equations and the {R}icci iteration on the space of
 {K}\"ahler metrics,} \emph{Adv. Math.} \textbf{218}, 5, pp. 1526--1565,
 \doi{10.1016/j.aim.2008.03.017}.

\bibitem[{Talenti(1976)}]{MR0463908}
Talenti, G. (1976). \enquote{Best constant in {S}obolev inequality,} \emph{Ann.
 Mat. Pura Appl. (4)} \textbf{110}, pp. 353--372.

\bibitem[{Trudinger(1968)}]{Tru}
Trudinger, N. (1968). \enquote{On imbeddings into orlicz spaces and some
 applications,} \emph{Indiana Univ. Math. J.} \textbf{17}, pp. 473--483.

\bibitem[{Villani(2009)}]{MR2459454}
Villani, C. (2009). \emph{Optimal transport, Old and new}, \emph{Grundlehren
 der Mathematischen Wissenschaften [Fundamental Principles of Mathematical
 Sciences]}, Vol. 338 (Springer-Verlag, Berlin), ISBN 978-3-540-71049-3,
 \doi{10.1007/978-3-540-71050-9}.

\end{thebibliography}

\end{document}